\documentclass[10pt]{amsart}
\usepackage{amscd,amssymb,hyperref}
\usepackage[PostScript=dvips]{diagrams}

\newtheorem{thm}{Theorem}[section]
\newtheorem{lem}[thm]{Lemma}
\newtheorem{cor}[thm]{Corollary}
\newtheorem{prop}[thm]{Proposition}

\def\square{\vbox{
      \hrule height 0.4pt
      \hbox{\vrule width 0.4pt height 5.5pt \kern 5.5pt \vrule width 0.4pt}
      \hrule height 0.4pt}}

\def\id{\mathrm{id}}
\def\even{\mathrm{even}}

\def\ch\mathrm{c h}
\def\ab{\mathrm{a b}}

\def\im{\mathrm{I m}}
\def\Hom{\mathrm{H o m}}

\def\End{\mathrm{End}}
\def\dim{\mathrm{dim}}
\def\End{\mathrm{End}}
\def\odd{\mathrm{odd}}

\def\hocolim{\operatorname{ho co l i m}}
\def\Lie{\mathrm{Lie}}

\def\proj{\mathrm{pro j}}

\def\GL{\mathrm{G L}}

\newcommand{\Z}{\mathbb{Z}}

\newcommand{\CoH}{\mathbf{CoH}}

\newcommand{\bfk}{\ensuremath{{\mathbf{k}}}}

\newcommand{\calS}{\ensuremath{\mathcal{S}}}

\newcommand{\calH}{\ensuremath{\mathcal{H}}}

\newcommand{\EHP}{\mathrm{EHP}}

\let\la=\langle
\let\ra=\rangle

\numberwithin{equation}{section}

\begin{document}

\newcommand{\auths}[1]{\textrm{#1},}
\newcommand{\artTitle}[1]{\textsl{#1},}
\newcommand{\jTitle}[1]{\textrm{#1}}
\newcommand{\Vol}[1]{\textbf{#1}}
\newcommand{\Year}[1]{\textrm{(#1)}}
\newcommand{\Pages}[1]{\textrm{#1}}

\address{Department of Mathematics\\
National University of Singapore\\
Singapore 119260\\
Republic of Singapore}\email{matwujie@math.nus.edu.sg}
\urladdr{http://www.math.nus.edu.sg/\~{}matwujie}

\author[J. Wu]{J. Wu$^{\dag}$}

\begin{abstract}
Let $X$ be a co-$H$-space of $(p-1)$-cell complex with all cells in even dimensions. Then the loop space $\Omega X$ admits a retract $\bar A^{\min}(X)$ that is the evaluation of the functor $\bar A^{\min}$ on $X$. In this paper, we determine the homology $H_*(\bar A^{\min}(X))$ and give the $\EHP$ sequence for the spaces $\bar A^{\min}(X)$.
\end{abstract}

\thanks{$^{\dag}$ Research is supported in part by the Academic Research Fund of the
National University of Singapore R-146-000-101-112.}



\title{The Functor $A^{\min}$ for $(p-1)$-cell Complexes and $\EHP$ Sequences}
\maketitle
\section{Introduction}

The algebraic functor $A^{\min}$ was introduced in~\cite{SW1}
arising from the question on the naturality of the classical
Poincar\'e-Birkhoff-Witt isomorphism. For any \textit{ungraded}
module $V$, $A^{\min}(V)$ is defined to be the \textit{smallest
functorial coalgebra retract} of $T(V)$ containing $V$. Then the
functor $A^{\min}$ extends canonically to the cases when $V$ is any
graded module. (See~\cite{SW1} for details.) The functor
$A^{\min}$ admits the tensor-length decomposition with
$$
A^{\min}(V)=\bigoplus_{n=0}^\infty A^{\min}_n(V),
$$
where $A^{\min}_n(V)=A^{\min}(V)\cap T_n(V)$ is the homogenous
component of $A^{\min}(V)$. By the Functorial
Poincar\'e-Birkhoff-Witt Theorem~\cite[Theorem 6.5]{SW1}, there
exists a functor $B^{\max}$ from (graded) modules to Hopf algebras
with the functorial coalgebra decomposition
$$
T(V)\cong A^{\min}(V)\otimes B^{\max}(V)
$$
for any graded module $V$. The determination of $A^{\min}(V)$ for
general $V$ is equivalent to an open problem in the modular
representation theory of the symmetric groups according
to~\cite[Theorem 7.4]{SW1}, which seems beyond the reach of current
techniques. Some properties of the algebraic functor $A^{\min}$ have been studied in~\cite{GW, Wu2} with applications in homotopy theory~\cite{GW2,Wu}. In this article we determine $A^{\min}(V)$ in the special cases
when $V_{\even}=0$ and $\dim V=p-1$.

Denote by $L(V)$ the free Lie algebra generated by $V$. Write $L_n(V)$ for the $n\,$th homogeneous
component of $L(V)$. Observe that $[L_s(V),L_t(V)]$ is a submodule
of $L_{s+t}(V)$ under the Lie bracket of $L(V)$. Let
$$
\bar L_n(V)=L_n(V)/\sum_{i=2}^{n-2} [L_i(V),L_{n-i}(V)].
$$
Define $\bar L_n^k(V)$ recursively by $\bar L_n^1(V)=\bar L_n(V)$
and $\bar L_n^{k+1}(V)=\bar L_n(\bar L_n^k(V))$.

\begin{thm}\label{theorem1.1}
Let the ground field $\bfk$ be of characteristic $p>2$. Let $V$ be a
graded module such that $V_{\even}=0$ and $\dim V=p-1$.  Then there
is an isomorphism of coalgebras
$$
A^{\min}(V)\cong \bigotimes_{k=0}^\infty E(\bar L_p^k (V)),
$$
where $E(W)$ is the exterior algebra generated by $W$.
\end{thm}

An important observation is that
$$
\bar L_p^k(V)_{\even}=0 \textrm { and } \dim \bar L_p^k(V)=p-1
$$
for each $k\geq0$ provided that $V_{\even}=0$ and $\dim V=p-1$.
Replace $V$ by $\bar L_p(V)$ in the above theorem. Then $
A^{\min}(\bar L_p(V))\cong \bigotimes\limits_{k=1}^\infty E(\bar
L_p^k (V))$ with a coalgebra decomposition
$$
A^{\min}(V)\cong E(V)\otimes A^{\min}(\bar L_p(V)),
$$
which indicates the existence of the $\EHP$ fibrations by taking the
\textit{geometric realization} of the functor $A^{\min}$ and the
\textit{Hopf invariants} on the functor $A^{\min}$.

The geometric realization of the (algebraic) functor $A^{\min}$ was studied in~\cite{GSW,STW,SW1, SW2, Theriault} for giving the decompositions of the loop spaces of double suspensions of torsion spaces. The most general result so far is given as
follows:

\begin{thm}~\cite[Theorem 1.2]{STW2}\label{theorem1.2}
There exist homotopy functors $\bar A^{\min}$ and $\bar B^{\max}$
from simply connected co-$H$ spaces of finite type to spaces such
that for any $p$-local simply connected co-$H$ space $Y$ the
following hold:
\begin{enumerate}
\item[1)] $\bar A^{\min}(Y)$ is a functorial retract of $\Omega Y$
and so there is a functorial decomposition
$$
\Omega Y\simeq \bar A^{\min}(Y)\times \bar B^{\max}(Y).
$$
\item[2)] On mod $p$ homology the decomposition
$$ H_*(\Omega Y)\cong H_*(\bar A^{\min}(Y))\otimes H_*(\bar B^{\max}(Y))$$ is with
respect to the augmentation ideal filtration.
\item[3)]On mod $p$ homology the associated bigraded $E^0H_*(\bar
A^{\min}(Y))$ is given by
$$
E^0H_*(\bar A^{\min}(Y))=A^{\min}(\Sigma^{-1}\bar H_*(Y)).
$$
\end{enumerate}
\end{thm}

Let $X$ be a path-connected finite complex. Define $
b_X=\sum\limits_{q=1}^\infty q\dim \bar H_q(X;\Z/p).$ Roughly
speaking, $b_X$ is the summation of the dimensions of the cells in
$X$. A direct consequence of Theorems~\ref{theorem1.1} and
~\ref{theorem1.2} is the following homological information:

\begin{cor}\label{corollary1.3}
Let $p>2$ and let $Y$ be any $p$-local simply connected co-$H$
space. Suppose that $\bar H_{\odd}(Y)=0$ and $\dim \bar H_*(Y)=p-1$.
Then there is an isomorphism of coalgebras
$$
H_*(\bar A^{\min}(Y))\cong E(\Sigma^{-1}\bar H_*(Y))\otimes
\bigotimes_{k=1}^\infty E(\Sigma^{\frac{p^k-1}{p-1}b_Y-p^k}\bar
H_*(Y)),
$$
where $\Sigma^{\frac{p^k-1}{p-1}b_Y-p^k}\bar H_*(Y)=\bar
L_p^k(\Sigma^{-1}\bar H_*(Y))$.\hfill $\Box$
\end{cor}

Recall that the classical Hopf invariants can be obtained from the
suspension splitting of the loop suspensions. For the geometric
functor $\bar A^{\min}$, the following suspension splitting
theorem is a special case of Theorem~\ref{theorem4.8}.

\begin{thm}\label{theorem1.4}
Let $Y$ be any $p$-local simply connected co-$H$ space. Then there
is a suspension splitting
$$
\Sigma \bar A^{\min}(Y)\simeq \bigvee_{n=1}^\infty \bar
A^{\min}_n(Y)
$$
such that
$$
\Sigma^{-1}\bar H_*(\bar A^{\min}_n(Y))\cong
A^{\min}_n(\Sigma^{-1}\bar H_*(Y))
$$
for each $n\geq 1$.\hfill $\Box$
\end{thm}

Note that each $\bar A^{\min}_n(Y)$ is a co-$H$ space because it is
a retract of $\Sigma \bar A^{\min}(Y)$. From the above theorem, one
gets the Hopf invariant $H_n$ defined as the composite
$$
\bar A^{\min}(Y)\rInto \Omega\Sigma \bar A^{\min}(Y)\simeq \Omega
\bigvee_{n=1}^\infty \bar A^{\min}_n(Y)\rTo^{\proj} \Omega \tilde
A^{\min}_n(Y)\rTo^r \bar A^{\min}(\bar A^{\min}_n(Y)),
$$
where $r\colon \Omega Z\to \bar A^{\min}(Z)$ is the functorial
retraction for any simply connected co-$H$-space $Z$. In particular,
there is a Hopf invariant $$ H_p\colon \bar A^{\min}(Y)\to \bar
A^{\min}(\bar A^{\min}_p(Y))$$ for any simply connected co-$H$ space
$Y$. By Lemma~\ref{lemma5.3},
$$\bar A^{\min}_p(Y)\simeq \Sigma^{b_Y-p+1}Y$$ provided that $\bar
H_{\odd}(Y)=0$ and
 $\dim \bar H_*(Y)=p-1$.
\begin{thm}[$\EHP$ Fibration]\label{theorem1.5}
Let $p>2$ and let $Y$ be any $p$-local simply connected co-$H$
space. Suppose that $\bar H_{\odd}(Y)=0$ and
 $\dim \bar H_*(Y)=p-1$.
Then there is a fibre sequence
$$
\Omega \bar A^{\min}(\Sigma^{b_Y-p+1}Y)\rTo^P \bar E(Y)\rTo^E \bar
A^{\min}(Y)\rTo^{H_p}\bar A^{\min}(\Sigma^{b_Y-p+1}Y).
$$
with the following properties:
\begin{enumerate}
\item[1)] On mod $p$ homology $H_*(\bar E(Y))\cong E(\Sigma^{-1}\bar H_*(Y))$
as coalgebras.
\item[2)] $H_*(\bar A^{\min}(Y))\cong H_*(\bar
E(Y))\otimes H_*(\bar A^{\min}(\Sigma^{b_Y-p+1}Y))$ as coalgebras.
\item[3)] If $f\colon S^n\to Y$ is a co-$H$ map such that $f_*\not=0\colon
\bar H_*(S^n)\to \bar H_*(Y)$. Then there is a commutative diagram
of fibre sequences
\begin{diagram}
\Omega \bar A^{\min}(\Sigma^{b_Y-p+1}Y)&\rTo^P& \bar E(Y)&\rTo^E&
\bar A^{\min}(Y)&\rTo^{H_p}&\bar A^{\min}(\Sigma^{b_Y-p+1}Y)\\
\uEq&& \uTo &&\uTo&&\uEq\\
\Omega \bar A^{\min}(\Sigma^{b_Y-p+1}Y)&\rTo^{P_f}& S^{n-1}&\rTo&
B_f& \rTo& \bar A^{\min}(\Sigma^{b_Y-p+1}Y).\\
\end{diagram}
In particular, the map $P\colon \Omega \bar
A^{\min}(\Sigma^{b_Y-p+1}Y)\to \bar E(Y)$ factors through the bottom
cell of $\bar E(Y)$.
\end{enumerate}
\end{thm}

The following theorem gives a general criterion when the $\EHP$
fibration splits off. A space $X$ is called to have a
\textit{retractile generating complex $C$} if $C$ is a retract of
$\Sigma X$ with a retraction $r\colon \Sigma X\to C$ such that the
mod $p$ cohomology $H^*(X)$ is generated by
$$
M=\im(\Sigma^{-1}r^*\colon \Sigma^{-1}\bar H^*(C)\to \bar H^*(X))
$$
and $M\cong QH^*(X)$, the set of indecomposables. Recall that a
space $X$ is called \textit{(stably) atomic} if any self (stable)
map of $X$ inducing isomorphism on the bottom homology is a (stable)
homotopy equivalence.

\begin{thm}\label{theorem1.6}
Let $p>2$ and let $Y$ be any $p$-local simply connected co-$H$ space
such that $Y$ is stably atomic. Suppose that
 $\bar H_{\odd}(Y)=0$ and $\dim \bar H_*(Y)=p-1$.
Then the following statements are equivalent to each other:
\begin{enumerate}
\item[1)] The $\EHP$ fibration
$$
\bar E(Y)\rTo^E \bar A^{\min}(Y)\rTo^{H_p}\bar
A^{\min}(\Sigma^{b_Y-p+1}Y)
$$
splits off.
\item[2)] $\bar E(Y)$ is an $H$-space.
\item[3)] There exists an $H$-space $X$ having $Y$ as a retractile
generating complex.
\item[4)] The map $P\colon \Omega \bar
A^{\min}(\Sigma^{b_Y-p+1}Y)\to \bar E(Y)$ is null homotopic.
\item[5)] The composite
$$
\Sigma^{b_Y-p-1}Y\rInto \Omega \bar
A^{\min}(\Sigma^{b_Y-p+1}Y)\rTo^P \bar E(Y)$$ is null homotopic.
\item[6)] There exists a map $g\colon \Sigma^{b_Y-p}Y\to \bar A^{\min}(Y)$ such that
$$
g_*\colon H_*(\Sigma^{b_Y-p}Y)\rTo H_*(\bar A^{\min}(Y))
$$
is a monomorphism on the bottom cells of $\Sigma^{b_Y-p}Y$.
\end{enumerate}
\end{thm}

It is a classical question whether the total space of a spherical fibration over a sphere is an $H$-space localized at an odd prime $p$. For having a possible $H$-space structure, the base space and the fibre must be odd dimensional spheres. For $p>3$, it is well-known that the answer is positive~\cite{CHZ,CN}. For $p=3$, there are examples that do not admit $H$-space structure. The above theorem gives some properties for studying this question.

The article is organized as follows. In section~\ref{section2}, we
study the representation theory on natural coalgebra decompositions
of tensor algebras. The proof of Theorem~\ref{theorem1.1} is given
in Section~\ref{section3}. The geometry of natural coalgebra
decompositions of tensor algebras is investigated in
Section~\ref{section4}, where Theorem~\ref{theorem1.4} is
Theorem~\ref{theorem4.8}. In Section~\ref{section5}, we give the
proof of Theorem~\ref{theorem1.5}. The proof of
Theorem~\ref{theorem1.6} is given in Section~\ref{section6}.

\section{The Functor $A^{\min}$ and the symmetric group module $\Lie(n)$}\label{section2}

\subsection{The Functor $A^{\min}$ on Ungraded Modules}\label{subsection2.1}
In this subsection, the ground ring is a field $\bfk$. A coalgebra
means a pointed coassociative cocommutative coalgebra. For any
module $V$, the tensor algebra $T(V)$ is a Hopf algebra by requiring
$V$ to be primitive. This defines $T\colon V\mapsto T(V)$ as a
functor from \textit{ungraded} modules to \textit{coalgebras}. The
functor $A^{\min}\colon V\mapsto A^{\min}(V)$ is defined to be the
smallest coalgebra retract of the functor $T$ with the property that
$V\subseteq A^{\min}(V)$. More precisely the functor $A^{\min}$ is
defined by the following property:
\begin{enumerate}
\item $A^{\min}$ is a functor from $\bfk$-modules to coalgebras with
a natural linear inclusion $V\rInto A^{\min}(V)$.
\item $A^{\min}(V)$ is a natural coalgebra retract of $T(V)$
over $V$. Namely there exist natural coalgebra transformations
$s\colon A^{\min}\to T$ and $r\colon T\to A^{\min}$ such that the
diagram
\begin{diagram}
A^{\min}(V)&\rTo^{s_V}& T(V)&\rTo^{r_V}&A^{\min}(V)\\
\uInto&&\uInto&&\uInto\\
V&\rEq& V&\rEq& V\\
\end{diagram}
commutes for any $V$ and $r\circ s=\id_{A^{\min}}$.
\item $A^{\min}$ is \textit{minimal} with respect to the above two conditions:
if $A$ is any functor from $\bfk$-modules to coalgebras with natural
linear inclusion $V\rInto A(V)$ such that $A(V)$ is a natural
coalgebra retract of $T(V)$ over $V$, then $A^{\min}(V)$ is a
natural coalgebra retract of $A(V)$ over $V$.
\end{enumerate}
By the minimal assumption, the functor $A^{\min}$ is unique up to
natural equivalence if it exists. The existence of the functor
$A^{\min}$ follows from~\cite[Theorem 4.12]{SW1}.

There is a multiplication on $A^{\min}$ given by the composite
$$
A^{\min}\otimes A^{\min}\rInto^{s\otimes s} T\otimes T\rTo^{\mu}
T\rOnto^r A^{\min}
$$
and so $A^{\min}$ is a functor from modules to quasi-Hopf algebras,
where a quasi-Hopf algebra means a coassociative and cocommutative
bi-algebra without assuming the associativity of the multiplication.
The multiplication on $A^{\min}$ induces a new natural coalgebra
transformation $r'_V\colon T(V)\to A^{\min}(V)$ given by
$$
r'_V(x_1\otimes\cdots\otimes x_n)=(((\cdots(x_1\cdot x_2)\cdot
x_3)\cdots)\cdot x_n)
$$
for any $x_1\otimes\cdots \otimes x_n\in V^{\otimes n}$. By the
minimal assumption, the composite
$$
A^{\min}\rTo^s T\rTo^{r'} A^{\min}
$$
is a natural equivalence. Consider $T(V)$ as an
$A^{\min}(V)$-comodule via the map $r'_V$. According
to~\cite[Proposition 6.1]{SW1}, the cotensor product
$$
B^{\max}(V)=\bfk\square_{A^{\min}(V)}T(V)
$$
is natural sub Hopf algebra of $T(V)$ with a natural coalgebra
equivalence
\begin{equation}
\bfk\otimes_{B^{\max}(V)}T(V)\cong A^{\min}(V).
\end{equation}
Together with~\cite[Lemmas 5.3]{SW1}, there is a natural coalgebra
decomposition
\begin{equation}
T(V)\cong B^{\max}(V)\otimes A^{\min}(V).
\end{equation}

By taking tensor length decomposition, we have $B^{\max}(V)=\bigoplus_{n=0}^\infty B^{\max}_n(V)$ with $B^{\max}_n(V)=B^{\max}(V)\cap T_n(V)$. Let
\begin{equation}
Q^{\max}(V)=QB^{\max}(V)
\end{equation}
be the indecomposables of $B^{\max}(V)$ with tensor length decomposition $Q^{\max}(V)=\bigoplus_{n=2}^\infty Q^{\max}_n(V)$. (Note that $B^{\max}_1(V)=0$ and so $B^{\max}(V)$ has no nontrivial indecomposable elements of tensor length $1$.) According to~\cite[Section 2]{SW1}, all of the functors $A^{\min}$, $A^{\max}_n$, $B^{\max}$, $Q^{\max}$ and $Q^{\max}_n$ extend canonically for graded modules.

\subsection{The Symmetric Group Module $\Lie(n)$}\label{subsection2.2}
Let the ground ring $R$ be any commutative ring with identity. The
module $\Lie^R(n)$, which is simply denoted as $\Lie(n)$ if the
ground ring is clear, is defined as follows. Let $\bar V$ be a free
$R$-module of rank $n$ with a basis $\{e_i\ | \ 1\leq i\leq n\}$.
The module $\gamma_n$ is defined to be the submodule of $\bar
V^{\otimes n}$ spanned by
$$
e_{\sigma(1)}\otimes e_{\sigma(2)}\otimes\cdots \otimes
e_{\sigma(n)}
$$
for $\sigma\in \Sigma_n$. The module $\Lie(n)$ is defined by
$$
\Lie(n)=\gamma_n\cap L_n(\bar V)\subseteq
\bar V^{\otimes n}.
$$
The $\Sigma_n$-action on $\gamma_n$ is given by permuting letters
$e_1,e_2,\ldots,e_n$. Since both $\gamma_n$ and $L_n(\bar V)$ are
invariant under permutations of the letters, $\Lie(n)$ is an
$R(\Sigma_n)$-submodule of $\gamma_n$.

Note that $\Lie(n)$ is the submodule of $L_n(\bar V)$ spanned by the
homogenous Lie elements in which each $e_i$ occurs exactly once. By
the Witt formula, $\Lie(n)$ is a free $R$-module of rank $(n-1)!$.
Following from the antisymmetry and the Jacobi identity, $\Lie(n)$
has a basis given by the elements
$$
[[e_1,e_{\sigma(2)}],e_{\sigma(3)}],\ldots,e_{\sigma(n)}]
$$
for $\sigma\in \Sigma_{n-1}$. (See ~\cite{Cohen2}.) Observe that $\Lie(n)$ is the image
of the $R(\Sigma_n)$-map
$$
\beta_n\colon \gamma_n\to\gamma_n \quad \beta_n(a_1\otimes\cdots
\otimes a_n)=[[a_1,a_2],\ldots,a_n].
$$
Thus $\Lie(n)$ can be also regarded as the quotient
$R(\Sigma_n)$-module of $\gamma_n$ with the projection
$\beta_n\colon \gamma_n\to \Lie(n)$.

\begin{prop}
Let $V$ be any graded projective module and let $\Sigma_n$ act on
$V^{\otimes n}$ by permuting factors in graded sense. Then there is
a functorial isomorphism
$$
\Lie(n)\otimes_{R(\Sigma_n)}V^{\otimes n}\cong L_n(V)
$$
for any graded module $V$.
\end{prop}
\begin{proof}
Clearly the quotient map $\beta_n\colon V^{\otimes
n}=\gamma_n\otimes_{R(\Sigma_n)}V^{\otimes n}\to L_n(V)$,
$x_1\otimes\cdots\otimes x_n\mapsto [[x_1,x_2],\ldots,x_n]$, factors
through the quotient $\Lie(n)\otimes_{R(\Sigma_n)}V^{\otimes n}$ and
so there is an epimorphism $\Lie(n)\otimes_{R(\Sigma_n)}V^{\otimes
n}\twoheadrightarrow L_n(V)$. On the other hand note that
$L=\bigoplus\limits_{n=1}^\infty
\Lie(n)\otimes_{R(\Sigma_n)}V^{\otimes n}$ has the canonical graded
Lie algebra structure generated by $V$. So the inclusion $V\rInto L$
induces an epimorphism of graded Lie algebras
$L(V)\twoheadrightarrow L$. The assertion follows.
\end{proof}

Consider $T_n\colon V\mapsto V^{\otimes n}$ as a functor from
projective (\textit{ungraded}) modules to projective
(\textit{ungraded}) modules. Denote by $\Hom(F,F')$ the set of
natural linear transformations from a functor $F$ to a functor $F'$
provided that $F$ preserves direct limits. By~\cite[Lemma 3.8]{GW},
$\Hom(T_n,T_m)=0$ if $n\not=m$ and there is an isomorphism of rings
\begin{equation}\label{theta}
\theta\colon \End_{R(\Sigma_n)}(\gamma_n)\rTo \End(T_n,T_n)
\end{equation}
given by
$$
\theta(\phi)=\phi\otimes_{\id_{V^{\otimes n}}}\colon
\gamma_n\otimes_{R(\Sigma_n)}V^{\otimes n}=V^{\otimes n}\rTo
\gamma_n\otimes_{R(\Sigma_n)}V^{\otimes n}=V^{\otimes n}
$$
for $\phi\in \End_{R(\Sigma_n)}(\gamma_n)$. Replacing $\gamma_n$
by $\Lie(n)$, we have the morphism of rings $ \theta\colon
\End_{R(\Sigma_n)}(\Lie(n))\rTo \End(L_n). $
\begin{prop}\label{proposition2.2}
If $n\not=m$, then $\Hom(L_n,L_m)=0$. Moreover the map $$
\theta\colon \End_{R(\Sigma_n)}(\Lie(n))\rTo \End(L_n)
$$
is an isomorphism.
\end{prop}
\begin{proof}
Let $\phi\colon L_n\to L_m$ be a natural transformation. Then the
composite $$T_n\rOnto L_n\rTo^\phi L_m\rInto T_m$$ is a natural
transformation, which is zero as $\Hom(T_n,T_m)=0$. Thus $\phi=0$.
For the second statement, clearly $\theta$ is a monomorphism. Let
$\phi_V\colon L_n(V)\to L_n(V)$ be any natural transformation. Let
$\bar V$ be the free $R$-module of rank $n$, which defines
$\gamma_n$. Consider the commutative diagram
\begin{diagram}
T_n(\bar V)&\rOnto^{\beta_n}& L_n(\bar V)&&\rTo^{\phi_{\bar V}}&& L_n(\bar V)&\rInto& T_n(\bar V)\\
\uInto&&\uInto&&&&\uInto&&\uInto\\
\gamma_n&\rOnto^{\beta_n}&\Lie(n)&&\rDashto^{\phi'}&&\Lie(n)&\rInto&\gamma_n,\\
\end{diagram}
where the existence of $\phi'$ follows from the fact that the
composite of the maps in the top row maps $\gamma_n$ into $\gamma_n$
and $\Lie(n)=\gamma_n\cap L_n(\bar V)$. Let $V$ be any ungraded
module and let $a_1\otimes\cdots \otimes a_n$ be any homogenous
element in $V^{\otimes n}$. Let $f\colon \bar V\to V$ be $R$-linear
map such that $f(e_i)=a_i$. By the naturality, there is a
commutative diagram
\begin{diagram}
\Lie(n)&\rInto & L_n(\bar V)&\rTo^{L_n(f)}& L_n(V)\\
\dTo>{\phi'}&&\dTo>{\phi_{\bar V}}&&\dTo>{\phi_V}\\
\Lie(n)&\rInto & L_n(\bar V)&\rTo^{L_n(f)}& L_n(V)\\
\end{diagram}
Thus
$$
\theta(\phi')([[a_1,a_2],\ldots,a_n])=\phi([[a_1,a_2],\ldots,a_n])
$$
and hence the result.
\end{proof}

\begin{cor}
There is a one-to-one correspondence between the decompositions of
the functor $L_n$ and the decompositions of $\Lie(n)$ over
$R(\Sigma_n)$.\hfill $\Box$
\end{cor}

A functor from modules to modules $Q$ is called
\textit{$T_n$-projective} if $Q$ is naturally equivalent to a
retract of the functor $T_n$.
\begin{prop}\label{projective}
Let $Q$ be a $T_n$-projective functor and let $\phi\colon Q\to L_n$
be a natural linear transformation. Then there exists a natural
linear transformation $\tilde\phi\colon Q\to T_n$ such that
$\phi=\beta_n\circ\tilde\phi$.
\end{prop}
\begin{proof}
Let $\bar V$ be the free $R$-module of rank $n$, which defines
$\gamma_n$. Let $\bar Q=\gamma_n\cap Q(\bar V)$. Since $Q$ is a
retract of the functor $T_n$, $\bar Q$ is a summand of $\gamma_n$
over $R(\Sigma_n)$ and so $\bar Q$ is a projective
$R(\Sigma_n)$-module. From the fact that
$\End_{R(\Sigma_n)}(\gamma_n)\cong \End(T_n)$,
$$
Q(V)=\bar Q\otimes_{R(\Sigma_n)}V^{\otimes n}
$$
for any module $V$. By the proof of Proposition~\ref{proposition2.2}, the map
$$
\theta\colon \Hom_{R(\Sigma_n)}(\bar Q,\Lie(n))\rTo \Hom(Q,L_n),
\quad f\mapsto f\otimes\id_{V^{\otimes n}}
$$
is an isomorphism. Thus there exists a unique $R(\Sigma_n)$-map
$$
\phi'\colon \bar Q\rTo \Lie(n)
$$
such that $\theta(\phi')=\phi$. Since $\bar Q$ is projective, the
lifting problem
\begin{diagram}
& &\gamma_n\\
&\ruDashto^{\tilde\phi'}&\dOnto>{\beta_n}\\
\bar Q&\rTo^{\phi'}&\Lie(n)\\
\end{diagram}
has a solution. The assertion follows by tensoring with $V^{\otimes
n}$ over $R(\Sigma_n)$.
\end{proof}

By inspecting the proof, each $T_n$-projective sub functor $Q$ of
$L_n$ induces a $R(\Sigma_n)$-projective submodule $\bar Q$ of
$\Lie(n)$. Conversely, each $R(\Sigma_n)$-projective submodule $\bar
Q$ of $\Lie(n)$ induces a $T_n$-projective sub functor $Q$, $
V\mapsto \bar Q\otimes_{R(\Sigma_n)}V^{\otimes n}$, of $L_n$. Thus
we have the following:
\begin{prop}\label{proposition2.5}
There is a one-to-one correspondence between $T_n$-projective sub
functors of $L_n$ and $R(\Sigma_n)$-projective submodules of
$\Lie(n)$.\hfill $\Box$
\end{prop}

\section{Proof of Theorem~\ref{theorem1.1}}\label{section3}

In this section, the ground field is of characteristic $p>2$. The
notation $V$ means any fixed connected graded module such that
$V_{\even}=0$ and $\dim V=p-1$. The general graded or ungraded module is then denoted by $W$.

Let $W$ be any module and let $T\colon W\mapsto T(W)$ be the functor from modules to Hopf algebras, where the tensor algebra $T(W)$ is Hopf by saying $W$ primitive.

Let $B(W)$ be the sub Hopf algebra generated
by $L_n(W)$ for $n$ not a power of $p$. By~\cite[Theorem 1.5]{SW1},  $B(W)\subseteq B^{\max}(W)$ for any
ungraded module $W$ and so for any graded or ungraded module $W$. It follows that there is an
epimorphism
\begin{equation}\label{equation3.1}
q\colon \bfk\otimes_{B(W)}T(W)\rOnto A^{\min}(W)
\end{equation}
for any ungraded or graded module $W$.

We are going to determine $\bfk\otimes_{B(V)}T(V)$ and to show that
the map $q$ is in fact an isomorphism when $W=V$.

\subsection{Determination of $\bfk\otimes_{B(V)}T(V)$}

Let $B^{[k]}(W)=\la B(W), L_{p^s}(W) \ | s\geq k\ra$ be the sub Hopf
algebra of $T(W)$ generated by $B(W)$ and $L_{p^s}(W)$ for $s\geq
k$. Then there is a tower of sub Hopf algebras
$$
\cdots \subseteq B^{[k+1]}(W)\subseteq B^{[k]}(W)\subseteq \cdots
\subseteq B^{[1]}(W)\subseteq B^{[0]}(W)=T(W)
$$
with the intersection
$$
B(W)=\bigcap_{k=0}^\infty B^{[k]}(W).
$$
Define the functor $\bar L_n$ by:
$$
\bar L_n(W)=L_n(W)/(\sum_{i=2}^{n-2}[L_i(W),L_{n-i}(W)])
$$
for each $n\geq 2$. According to~\cite[Proposition 11.6]{SW1}, $Q_{p}B^{[1]}(W)=\bar L_p$. So we can
identify $\bar L_p$ with a sub functor of $L_p$. Now define
recursively the functor $\bar L_p^k$ by $\bar L_p^0(W)=W$, $\bar
L_p^1(W)=\bar L_p(W)$ and $\bar L_p^{k+1}(W)=\bar L_p(\bar
L_p^k(W))$. Note that $\bar L_p^k$ has the tensor length $p^k$ and $\bar L_p^k$ can be identified with a sub functor of $L_{p^k}$.
Observe that $(\bar L_p(V))_{\even}=0$ and $\dim \bar L_p(V)=p-1$.
Thus we have the periodicity information:
\begin{equation}\label{equation3.3}
(\bar L_p^k(V))_{\even}=0  \textrm { and } \dim \bar L_p^k(V)=p-1
\end{equation}
for all $k\geq 0$.

\begin{lem}~\cite[Lemma 2.37]{Wu}\label{lemma3.1}
Let $A$ be a connected Hopf algebra of finite type and let $B$ be a
sub Hopf algebra of $A$. Suppose that $A$ is a tensor algebra as an
algebra with a choice of inclusion $QA\to A$. Then there is a short
exact sequence
$$
0\rTo Q(B)\rTo (k\otimes_BA)\otimes Q(A)\rTo I(k\otimes_BA)\rTo0.
$$
\end{lem}

Now we determine the sub Hopf algebra $B^{[k]}(V)$:

\begin{lem}\label{lemma3.2}
For each $k\geq 0$, there is a short exact sequence of Hopf algebras
$$
B^{[k+1]}(V)\rInto B^{[k]}(V)\rOnto E(\bar L_p^k(V)).
$$
Moreover
\begin{enumerate}
\item $B^{[k]}(V)$ is the sub Hopf algebra of $T(V)$ generated by $\bar
L_p^k(V)$ and $Q_jB(V)$ for $2\leq j<p^k$.
\item For any possible Steenrod module structure on $V$, $\bar L_p^k(V)$ is a suspension of
$V$.
\end{enumerate}
\end{lem}
\begin{proof}
The proof is given by induction. First consider the short exact
sequence of Hopf algebra
$$
B'(V)\rInto T(V)\rOnto \Lambda(V),
$$
where $B'(V)=\bfk\square_{\Lambda(V)}T(V)$. Then
$B^{[1]}(V)\subseteq B'(V)$. There is a short exact sequence
$$
0\rTo Q(B'(V))\rTo \Lambda(V) \otimes V\rTo I\Lambda(V)\rTo0.
$$
Since $V_{\even}=0$, $E(V)=\Lambda(V)$. By the assumption of $\dim
V=p-1$, we have $E(V)_j=0$ for $j\geq p$. Thus $Q_j(B'(V))=0$ for
$j>p$ and so $B'(V)\subseteq B^{[1]}(V)$. Hence $B'(V)=B^{[1]}(V)$
From the above exact sequence,
$$
Q_p(B'(V))\cong E_{p-1}(V)\otimes V\cong Q_p(B^{[1]}(V))=\bar L_p(V)
$$
is of dimension $p-1$. Since $\dim E_{p-1}(V)=1$, $\bar
L_p(V)=E_{p-1}(V)\otimes V$ is a suspension of $V$ for any possible
Steenrod module structure on $V$. Thus the assertions hold for
$k=0$. Suppose that the assertion holds for $k$. Consider the short
exact sequence of Hopf algebras
$$
B''(V)\rInto B^{[k]}(V)\rOnto^{\phi} \Lambda(\bar L_p^k(V)),
$$
where $B''(V)=\bfk\square_{\Lambda(\bar L_p^k(V))}B^{[k]}(V)$. Then
$B^{[k+1]}(V)\subseteq B''(V)$ because $\phi$ is a Hopf map which
sends the generators for $B^{[k+1]}(V)$ to zero. Since $\bar
L_p^k(V)_{\even}=0$ with $\dim \bar L_p^k(V)=p-1$, $\Lambda(\bar
L_p^k(V))=E(\bar L_p^k(V))$ and so there is short exact sequence
$$
0\rTo Q(B''(V))\rTo E(\bar L_p^k (V)) \otimes Q(B^{[k]}(V)\rTo
IE(\bar L_p^k(V))\rTo0.
$$
It follows that $Q(B''(V))_j=0$ for $j>p^{k+1}$ and so
$B''(V)\subseteq B^{[k+1]}(V)$. By the Lie action of $E(\bar
L_p^k(V))$ on $Q(B''(V))$,
$$
Q_{p^{k+1}}B^{[k+1]}(V)=Q_{p^{k+1}}B''(V)=\bar L_p^{k+1}(V)\cong
E_{p-1}(\bar L_p^k(V))\otimes \bar L_p^k(V).
$$
Since $\dim E_{p-1}(\bar L_p^k(V))=1$, $\bar L_p^{k+1}(V)$ is a
suspension of $\bar L_p^k(V)$. The induction is finished and hence
the result.
\end{proof}

\begin{prop}\label{proposition3.3}
There is a (non-functorial) coalgebra decomposition
$$
\bfk\otimes_{B(V)}T(V)\cong \bigotimes_{k=0}^\infty E(\bar L_p^k(V))
$$
over any possible Steenrod algebra structure on $V$.
\end{prop}
\begin{proof}
Consider the short exact sequence of coalgebras
\begin{equation}\label{equation3.2}
\bfk\otimes_{B(V)}B^{[k+1]}(V)\rInto^j
\bfk\otimes_{B(V)}B^{[k]}(V)\rOnto E(\bar L_p^k(V)).
\end{equation}
Observe that up to tensor length $p^{k+1}-1$
$$
(\bfk\otimes_{B(V)}B^{[k]}(V))_j\cong E(\bar L_p^k(V))_j
$$
for $j\leq p^{k+1}$. Thus there is a coalgebra cross-section
$$
s\colon E(\bar L_p^k(V))\rInto \bfk\otimes_{B(V)}B^{[k]}(V).
$$
The cross-section $s$ is a morphism over the Steenrod algebra for
all possible Steenrod algebra structure on $V$ and it is also a
morphism over $\GL(V)$ when $\GL(V)$ acts on $V$ by forgetting the
grading of $V$. 

According to~\cite[Theorem 1.1]{LW2}, $B(W)$ is a functorial coalgebra retract of $T(W)$ 
for any graded module $W$. Let $r_W\colon T(W)\to B(W)$ be a functorial coalgebra retraction. 
From $B(W)\subseteq B^{[k]}(W)\subseteq T(W)$, the restriction 
$$r_V|_{B^{[k]}(W)}\colon B^{[k]}(V)\rTo B(W)$$ is 
a functorial coalgebra retraction. It follows that the short exact sequence of coalgebras
$$
B(W)\rInto B^{[k]}(W)\rOnto \bfk\otimes_{B(W)} B^{[k]}(W)
$$
splits off. This gives a functorial coalgebra decomposition
$$
B^{[k]}(W)\cong B(W)\otimes (\bfk\otimes_{B(W)} B^{[k]}(W))
$$
for any graded module $W$. Thus the functor
$$
W\mapsto \bfk\otimes_{B(W)}B^{[k]}(W)
$$
is a functorial from modules to quasi-Hopf algebra because $\bfk\otimes_{B(W)}B^{[k]}(W)$ is a functorial coalgebra retract of the Hopf algebra functor $B^{[k]}(W)$.

Evaluating $W=V$, $\bfk\otimes_{B(V)}B^{[k]}(V)$ is a quasi-Hopf
algebra over the Steenrod algebra for
all possible Steenrod algebra structure on $V$. By ~\cite[Lemma 5.3]{SW1}, the map
$$
(\bfk\otimes_{B(V)}B^{[k+1]}(V))\otimes E(\bar
L_p^k(V))\rTo^{\mu(j\otimes s)} \bfk\otimes_{B(V)}B^{[k]}(V)
$$
is an isomorphism and hence the result.

\end{proof}

\subsection{Proof of Theorem~\ref{theorem1.1}}
Consider the functorial short exact sequence of algebras
$$
B^{\max}(W)\rTo T(W)\rTo^r A^{\min}(W)=\bfk\otimes_{B^{\max}}T(W).
$$
Define
$$
A^{\min}_n(W)=\im( T_n(W)\rInto T(W)\rTo^r A^{\min}(W).
$$
Then
$$
A^{\min}(W)=\bigoplus_{n=0}^\infty A^{\min}_n(W)
$$ is a bigraded coalgebra and the decomposition
$$
T(W)\cong B^{\max}(W)\otimes A^{\min}(W)
$$
is a functorial coalgebra decomposition of bi-graded coalgebras,
where the second grading on $B^{\max}(W)$ is given by
$B^{\max}_n(W)=T_n(W)\cap B^{\max}(W).$ Note that
$$
L_n(W)=\im(\beta_n\colon W^{\otimes n}\to W^{\otimes n}).
$$
Denote by $\beta_n$ the epimorphism $W^{\otimes n}\twoheadrightarrow
L_n(W)$ if there are no confusions. The algebraic version of the
James-Hopf map $H_n\colon T(W)\rTo T(W^{\otimes n})$ is defined
in~\cite{GW,SW1,SW2}.

\begin{lem}\label{lemma3.4}
There is a commutative diagram of natural transformations of
functors
\begin{diagram}
& & T&\rTo^{H_{p^k}}& T(T_{p^k})& \rTo^{T(r)}& T(A^{\min}_{p^k})&\rOnto &A^{\min}(A^{\min}_{p^k})\\
&&\uInto&&\uInto&&\uInto &&\uTo\\
Q^{\max}_{p^{k+1}}&\rInto& L_{p^{k+1}}&\rTo& L_p(T_{p^k})&
\rTo^{L_p(r)}&
L_p(A^{\min}_{p^k})&\rOnto &\bar L_p(A^{\min}_{p^k})\\
\uEq&& \uOnto>{\beta_{p^k}}&&\uOnto>{\beta_p}&&\uOnto>{\beta_p}&&\\
Q^{\max}_{p^{k+1}}&\rTo^{\phi'}& T_{p^{k+1}}&\rTo^{\phi}&
T_p(T_{p^k})&\rTo^{T_p(r)}& T_p(A^{\min}_{p^k})&&\\
\end{diagram}
for some $\phi$ and $\phi'$.
\end{lem}
\begin{proof}
Since all functors in the diagram are well-defined over $\Z_{(p)}$,
it suffices to show that the assertion holds when the ground ring is
$\Z_{(p)}$. The top three squares commute because the maps in the
top row are coalgebra maps, which send the primitives to the
primitives. It is clear that the bottom right square commutes. By
Proposition~\ref{projective}, there exist lifting $\phi'$ and $\phi$
such that the bottom left and middle squares commute.
\end{proof}

\begin{lem}\label{lemma3.5}
Let $n\geq 2$ and let $W$ be any graded module. Then the composite
$$
W^{\otimes n}=W^{
\otimes 2}\otimes W^{\otimes n-2}\rOnto^{\beta_n} L_n(W)\rOnto \bar L_n(W)
$$
factors through the quotient
$$
\left(W^{\otimes 2}/\la a_1\otimes a_2+(-1)^{|a_1||a_2|}a_2\otimes
a_1\ra\right)\otimes \Lambda_{n-2}(W)
$$
of $W^{\otimes n}$.
\end{lem}
\begin{proof}
By the skew-symmetric property on the first two factors, the
iterated Lie operad $\beta_n\colon W^{\otimes n}\to L_n(W)$ factors
through the quotient
$$
\left(W^{\otimes 2}/\la a_1\otimes a_2+(-1)^{|a_1||a_2|}a_2\otimes
a_1\ra\right)\otimes W^{\otimes n-2}.
$$
Consider the short exact
sequence of Hopf algebras
$$
B^{[1]}(W)\rInto T(W)\rOnto \Lambda(W).
$$
Note that there is a commutative diagram
\begin{diagram}
\bigoplus_{n=2}^\infty L_n(W)&\rInto& IB^{[1]}(W)\\
\dOnto&&\dOnto\\
\bigoplus_{n=2}^\infty \bar L_n(W)&\rEq& QB^{[1]}(W).\\
\end{diagram}
The assertion follows from~\cite[Lemma 3.12]{CMN} that the Lie
bracket in $T(W)$ induces an action of $\Lambda (W)$ on
$QB^{[1]}(W)=\bigoplus_{n=2}^\infty \bar L_n(W)$.
\end{proof}

Let $W$ be a graded module such that $W_j=0$ for $j\not=1$. Let
$\Sigma_n$ act on $W^{\otimes n}$ by permuting positions in graded
sense. Let $\GL(W_1)$ be the general linear group of the vector
space $W_1$. The Schur algebra~\cite{Schur} is defined by
$$
\calS(W)=\im(\phi\colon \bfk(\GL(W))\to \End(W^{\otimes n})),
$$
where $\phi(f)=f^{\otimes n}.$
By~\cite[Theorem 4.1]{CR},
\begin{equation}\label{equation3.4}
\im(\bfk(\Sigma_n)\to \End (W^{\otimes
n}))=\End_{\calS(W)}(W^{\otimes n})
\end{equation}
if $\bfk$ is algebraically closed. Thus if $\bfk$ is algebraically
closed, then for any $\bfk(\GL(W))$-map $$f\colon W^{\otimes n}\to
W^{\otimes n}$$ there exists an element $\alpha\in \bfk(\Sigma_n)$
such that
$$
f(a_1\otimes\cdots\otimes a_n)=\alpha\cdot (a_1\otimes\cdots\otimes a_n).
$$

\begin{lem}\label{lemma3.6}
Let $\bar V$ be a graded module such that $\dim\bar V_1=p-1$ and
$\bar V_j=0$ for $j\not=1$. Let $\GL(\bar V)$ be the general linear
group with the action on $\bar V$. Assume that the ground field
$\bfk$ is algebraically closed. Then, as a morphism over
$\bfk(\GL(\bar V))$, the composite
$$
q\colon T_p(\bar V)=\bar V^{\otimes p}\rOnto^{\beta_p} L_p(\bar
V)\rOnto \bar L_p(\bar V)\rTo^{\cong} \bar V\otimes E_{p-1}(\bar V)
$$
does NOT have a cross-section.
\end{lem}
\begin{proof}
Suppose that there exists a $\bfk(\GL(\bar V))$-map $\phi\colon \bar
V\otimes E_{p-1}(\bar V)\rTo \bar V^{\otimes p}$ such that the
composite $q\circ\phi$ is the identity map of $\bar V\otimes
E_{p-1}(\bar V)$. We are going to find a contradiction.

Let $\{x_1,x_2,\ldots,x_{p-1}\}$ be a basis for $\bar V$. Note that
$$\dim E_{p-1}(\bar V)=1$$ with a basis $\{(x_1\wedge x_2\wedge \cdots \wedge x_{p-1})\}$.
The set
$$
\{x_i\otimes (x_1\wedge x_2\wedge \cdots\wedge x_{p-1}) \ | \ 1\leq
i\leq p-1\}
$$
forms a basis for $\bar V\otimes E_{p-1}(\bar V)$. Since the
composite
$$
\bar V^{\otimes p} \rOnto^q \bar V\otimes E_{p-1}(\bar V)\rTo^\phi
\bar V^{\otimes p}
$$
is an endomorphism over $\bfk(\GL(\bar V))$, from
equation~(\ref{equation3.4}) there exists an element $\alpha\in
\bfk(\Sigma_p)$ such that
\begin{equation}\label{equation3.5}
\phi\circ q(w)=\alpha\cdot w
\end{equation}
for $w\in \bar V^{\otimes p}$. By counting the occurrence of
$x_j$'s, $\phi(x_i\otimes (x_1\wedge \cdots \wedge x_{p-1}))$ is a
linear combination of monomials $x_{i_1}x_{i_2}\cdots x_{i_p}$ in
which $x_i$ occurs twice and $x_j$ occurs once for each $j\not=i$.
By Lemma~\ref{lemma3.5}, there is a commutative diagram
\begin{diagram}
\bar V^{\otimes p}&\rOnto^{\beta_p}&&L_p(\bar V)\\
\dOnto>{\pi}&&&\dOnto>{\bar\pi}\\
\left( \bar V^{\otimes 2}/\la x_ix_j-x_jx_i\ra\right)\otimes
E_{p-2}(\bar V)&\rTo^g& &\bar L_p(\bar V)=\bar V\otimes E_{p-1}(\bar V).\\
\end{diagram}
By counting the occurrence of $x_j$'s,
\begin{equation}\label{equation3.6}
\begin{array}{c}
\pi\circ\phi(x_1\otimes (x_1\wedge x_2\wedge \cdots \wedge
x_{p-1}))\\
=k_1x_1^2\otimes (x_2\wedge x_3\wedge \cdots\wedge x_{p-1})\\
+\sum_{i=2}^{p-1}k_ix_1x_i\otimes (x_1\wedge\cdots\wedge
x_{i-1}\wedge x_{i+1}\wedge\cdots \wedge x_{p-1}).\\
\end{array}
\end{equation}
for some $k_i\in \bfk$. Note that $x_i\wedge x_j=-x_j\wedge x_i$. By
interchanging $x_i$ and $x_{i+1}$, there is an equation
$$
k_i=-k_{i+1}
$$
for $i\geq 2$. By equation~(\ref{equation3.5}), the map
$$
\pi\circ\phi\colon \bar V\otimes E_{p-1}(\bar V)\rTo \left( \bar
V^{\otimes 2}/\la x_ix_j-x_jx_i\ra\right)\otimes E_{p-2}(\bar V)
$$
is a morphism over the Steenrod algebra by regarding $\bar V$ as
\textit{any graded module over the Steenrod algebra} with $\bar
V_{\textrm{even}}=0$. Take the Steenrod module structure on $\bar V$
by setting $P^1_*x_1=x_2$, $P^i_*x_1=0$ for $i>1$ and $P^k_*x_j=0$
for $k\geq 1$ and $j>1$. Then
$$
\begin{array}{l}
P^2_*x_1x_i\otimes (x_1\wedge\cdots\wedge x_{i-1}\wedge
x_{i+1}\wedge\cdots \wedge x_{p-1})\\ =x_2x_i\otimes
(x_2\wedge\cdots\wedge x_{i-1}\wedge x_{i+1}\wedge\cdots \wedge
x_{p-1})\\
=0\\
\end{array}
$$
for $i\geq3$.
Note that
$$
P^2_*x_1\otimes (x_1\wedge x_2\wedge \cdots \wedge
x_{p-1})=x_2\otimes (x_2\wedge x_2\wedge \cdots \wedge x_{p-1})=0.
$$
By applying $P^2_*$ to equation~(\ref{equation3.6}),
$$
\begin{array}{lll}
0&=&P^2_*(k_1x_1^2\otimes (x_2\wedge x_3\wedge \cdots\wedge
x_{p-1})+k_2x_1x_2\otimes (x_1\wedge x_3\wedge \cdots\wedge
x_{p-1}))\\
&=&k_1x_2^2\otimes (x_2\wedge x_3\wedge \cdots\wedge x_{p-1})+k_2
x_2^2 \otimes (x_2\wedge x_3\wedge \cdots\wedge x_{p-1}).\\
\end{array}
$$
Thus
$$
k_1=-k_2
$$
and so
$$
\pi\circ\phi(x_1\otimes (x_1\wedge x_2\wedge \cdots \wedge
x_{p-1}))= k_1 \sum_{i=1}^{p-1}(-1)^{i-1}x_1x_i\otimes
(x_1\wedge\cdots\wedge x_{i-1}\wedge x_{i+1}\wedge\cdots \wedge
x_{p-1}).
$$

Define
$$
\alpha_1= x_1\otimes x_1\otimes x_2\otimes x_3\otimes \cdots \otimes
x_{p-1}\in \bar V^{\otimes p}
$$
$$
\alpha_i= x_1\otimes x_i\otimes x_1\otimes \cdots \otimes
x_{i-1}\otimes x_{i+1}\otimes\cdots \otimes x_{p-1}\in \bar
V^{\otimes p}
$$
for $2\leq i\leq p-1$. Then
$$
\begin{array}{ll}
\bar\pi\beta_p(\alpha_1)&=\bar\pi([x_1,x_1],\ldots,x_{p-1}])\\
 &=2x_1\otimes (x_1\wedge x_2\wedge x_3\wedge\cdots\wedge x_{p-1}).\\
\end{array}
$$
For $i\geq 2$,
$$
\begin{array}{ll}
\bar\pi\beta_p(\alpha_i)&=\bar\pi([x_1,x_i],x_1],\ldots,x_{i-1}],x_{i+1}],x_{p-1}])\\
 &=(-1)^{i-1}x_1\otimes (x_1\wedge x_2\wedge x_3\wedge\cdots\wedge x_{p-1}).\\
\end{array}
$$
It follows that
$$
\begin{array}{l}
q\circ \phi(x_1\otimes (x_1\wedge x_2\wedge \cdots \wedge
x_{p-1}))\\
= g\circ \pi\circ \phi(x_1\otimes (x_1\wedge x_2\wedge \cdots \wedge
x_{p-1}))\\
=k_1 g(\sum\limits_{i=1}^{p-1}(-1)^{i-1}x_1x_i\otimes
(x_1\wedge\cdots\wedge x_{i-1}\wedge x_{i+1}\wedge\cdots \wedge
x_{p-1}))\\
=k_1 g(\sum\limits_{i=1}^{p-1}(-1)^{i-1}\pi(\alpha_i))\\
=k_1\sum\limits_{i=1}^{p-1}(-1)^{i-1}g\circ\pi(\alpha_i)\\
=k_1(\bar\pi\beta_p(\alpha_1)+\sum\limits_{i=2}^{p-1}(-1)^{i-1}\bar\pi\beta_p(\alpha_i)\\
=pk_1x_1\otimes (x_1\wedge x_2\wedge x_3\wedge\cdots\wedge
x_{p-1})\\
=0,\\
\end{array}
$$
which contradicts to the assumption that $q\circ \phi$ is the
identity map. This finishes the proof.
\end{proof}

\begin{proof}[Proof of Theorem~\ref{theorem1.1}]
By Proposition~\ref{proposition3.3}, it suffices to show that the
epimorphism
$$
\bfk\otimes_{B(V)}T(V)\rTo A^{\min}(V)
$$
is an isomorphism.

Since the functor $A^{\min}$ only depends on the characteristic of
the ground field, we may assume that the ground field $\bfk$ is
algebraically closed. Denote by $\phi_k$ the epimorphism $
\bfk\otimes_{B(V)}T(V)\rTo A^{\min}(V) $. By
Proposition~\ref{proposition3.3},
$$
\bar L_p^k(V)=P_{p^k}(\bfk\otimes_{B(V)}T(V))
$$
for each $k\geq 0$. By equation~(\ref{equation3.1}), $A^{\min}(V)$
is functorial coalgebra retract of $\bfk\otimes_{B(V)}T(V)$. It
suffices to show by induction that the retraction map
$$
\phi_k|_{\bar L_p^k(V)}\colon \bar
L_p^k(V)=P_{p^k}(\bfk\otimes_{B(V)}T(V))\to P_{p^k}A^{\min}(V)
$$
is an isomorphism for each $k\geq0$. The statement holds for $k=0$.
Suppose that the statement holds for $s\leq k$ with $k>0$ and
consider
$$
\phi_{k+1}|_{\bar L_p^{k+1}(V)}\colon \bar
L_p^{k+1}(V)=P_{p^{k+1}}(\bfk\otimes_{B(V)}T(V))\to
P_{p^{k+1}}A^{\min}(V).
$$

By choosing $V$ to a module over the Steenrod algebra with a basis
given by $\{v,P^1_*v,P^2_*v,\ldots,P^{p-2}_*v\}$, then $V$ is an
indecomposable module over the Steenrod algebra. By
Lemma~\ref{lemma3.2}, $\bar L_p^{k+1}(V)$ is a suspension of $V$ as
a module over the Steenrod algebra. Thus $\bar L_p^{k+1}(V)$ is an
indecomposable module over the Steenrod algebra. It follows that the
retraction
$$
\phi_{k+1}|_{\bar L_p^{k+1}(V)}\colon \bar
L_p^{k+1}(V)=P_{p^{k+1}}(\bfk\otimes_{B(V)}T(V))\to
P_{p^{k+1}}A^{\min}(V)
$$
is either identically zero or an isomorphism.

Suppose that $\phi_{k+1}|_{\bar L_p^{k+1}(V)}=0$. Then
$$
P_{p^{k+1}}A^{\min}(V)=0
$$
and so
$$
L_{p^{k+1}}(V)\subseteq B^{\max}(V).
$$
By Lemma~\ref{lemma3.2}, $B^{[k+1]}(V)$ is the sub Hopf algebra of
$T(V)$ generated by $\bar L_p^{k+1}(V)$ and $L_n(V)$ for $n$ not a
power of $p$. It follows that
\begin{equation}\label{equation3.7}
B^{[k+1]}(V)\subseteq B^{\max}(V)
\end{equation}
and so
$$
PA^{\min}(V)\subseteq \bigoplus_{s=1}^k \bar L_p^s(V).
$$
From the induction. $\bar L^s_p(V)\subseteq A^{\min}(V)$ for each
$s\leq k$. Thus
$$
PA^{\min}(V)= \bigoplus_{s=1}^k \bar L_p^s(V).
$$
It follows that there is a (\textit{non-functorial}) isomorphism of
coalgebra
$$
A^{\min}(V)\cong \bigotimes_{s=0}^k E(\bar L_p^s (V)).
$$
By computing Poincar\'e series, the inequality~(\ref{equation3.7})
becomes the equality $$B^{[k+1]}(V)=B^{\max}(V).$$ Thus
$$
Q^{\max}_{p^{k+1}}(V)=Q_{p^{k+1}}B^{\max}(V)=Q_{p^{k+1}}B^{[k+1]}(V)\cong
\bar L_p^{k+1}(V).
$$
Observe that the inclusion
$$
\bar L_p^k(V)=P_{p^k}A^{\min}(V)\rInto A^{\min}_{p^k}(V)
$$
is an isomorphism and $\bar L_p^{k+1}(V)=\bar L_p(\bar L_p^k(V))$.
The composite of natural transformations
\begin{equation}\label{equation3.8}
Q^{\max}_{p^{k+1}}\rInto L_{p^{k+1}}\rTo L_p(T_{p^k}) \rTo^{L_p(r)}
L_p(A^{\min}_{p^k}) \rOnto \bar L_p(A^{\min}_{p^k})
\end{equation}
becomes an isomorphism by evaluating on the graded module $V$. Write
$\bar V$ for $ \bar L_p^k(V)=P_{p^k}A^{\min}(V)= A^{\min}_{p^k}(V).$
By Lemma~\ref{lemma3.4}, the composite of the first three natural
transformations in equation~(\ref{equation3.8}) admits a lifting of
natural transformation into $T_p(A^{\min}_{p^k})$ via the
epimorphism $\beta_p\colon T_p(A^{\min}_{p^k})\rOnto
L_p(A^{\min}_{p^k})$ and so the composite
$$
q\colon T_p(\bar V)=\bar V^{\otimes p}\rOnto^{\beta_p} L_p(\bar
V)\rOnto \bar L_p(\bar V)\rTo^{\cong} \bar V\otimes E_{p-1}(\bar V)
$$
admits a cross-section over the general linear group algebra
$\bfk(\GL(\bar V))$ by forgetting the grading of $\bar V$. This is
impossible according to Lemma~\ref{lemma3.6}.

Thus $\phi_{k+1}|_{\bar L_p^{k+1}(V)}\not=0$ and so
$\phi_{k+1}|_{\bar L_p^{k+1}(V)}$ must be an isomorphism. The
induction is finished and hence the result.
\end{proof}

\section{The Geometry of Natural Coalgebra Decompositions}\label{section4}
\subsection{Geometric Realizations}
Denote by $\CoH$ the category of $p$-local simply connected
co-$H$-spaces of finite type and co-$H$-maps. Let
$[\Omega,\Omega]_{\CoH}$ be the group of natural transformations of
the functor $\Omega$ from $\CoH$ to the homotopy category of spaces.

\begin{thm}[Geometric Realization Theorem]~\cite{STW2}\label{theorem4.1}
Let $Y$ be any simply connected co-$H$-space of finite type and let

$$
T(V)\cong A(V)\otimes B(V)
$$
any natural coalgebra decomposition for ungraded modules over
$\Z/p$. Then there exist homotopy functors $\bar A$ and $\bar B$
from $\CoH$ to spaces such that
\begin{enumerate}
\item there is a functorial decomposition
$$
\Omega Y\simeq \bar A(Y)\times \bar B(Y)
$$
\item On mod $p$ homology the decomposition $$H_*(\Omega Y)\cong H_*(\bar
A(Y))\otimes H_*(\bar B(Y))$$ is with respect to the augmentation
ideal filtration
\item On mod $p$ homology
$$
E^0H_*(\bar A(Y))=A(\Sigma^{-1}\bar H_*(Y)) \textrm{ and }
E^0_*(\bar B(Y))=B(\Sigma^{-1}\bar H_*(Y)),
$$
where $A$ and $B$ are the canonical extensions of the functors $A$
and $B$ for graded modules\hfill $\Box$
\end{enumerate}
\end{thm}

For the functors $A^{\min}$ and $B^{\max}$, we have the following geometric realization theorem.

\begin{thm}~\cite{STW2}\label{theorem4.2}
There exist homotopy functors $\bar A^{\min}$, $\bar Q^{\max}_n$, $n\geq2$,
from $\CoH$ to spaces such that for any $p$-local simply connected
co-$H$ space $Y$ of finite type the following hold:
\begin{enumerate}
\item[1)] $\bar B^{\max}(Y)=\Omega\left(\bigvee_{n=2}^\infty \bar Q^{\max}_n(Y)\right)$.
\item[2)] There is a functorial decomposition
$$
\Omega Y\simeq \bar A^{\min}(Y)\times \bar B^{\max}(Y).
$$
\item[3)] On mod $p$ homology the decomposition
$$ H_*(\Omega Y)\cong H_*(\bar A^{\min}(Y))\otimes H_*(\bar B^{\max}(Y))$$ is with
respect to the augmentation ideal filtration.
\item[4)]On mod $p$ homology
$$
\begin{array}{c}
E^0H_*(\bar A^{\min}(Y))=A^{\min}(\Sigma^{-1}\bar H_*(Y)),\\
E^0H_*(\bar B^{\max}(Y))=B^{\max}(\Sigma^{-1}\bar H_*(Y)),\\
E^0\bar H_*(\bar Q^{\max}(Y))=Q^{\max}(\Sigma^{-1}\bar H_*(Y)).\\
\end{array}
$$\hfill $\Box$
\end{enumerate}
\end{thm}

\subsection{Suspension Splitting Theorems}
In this subsection, we review the suspension splitting
theorems.

A \textit{graded space} means a space $W$ with a homotopy
decomposition
$$
\phi_W\colon W\rTo^{\simeq}\bigvee_{n=1}^\infty W_n.
$$
For any graded space $W$, the homology $\bar H_*(W)$ is filtered by
$$
I^t\bar H_*(W)=\phi_*^{-1}(\bar H_*(\bigvee_{n=t}^\infty W_n))
$$
for $t\geq 1$. A \textit{graded co-$H$ space} means a graded space
$W$ such that $W$ is a co-$H$ space. Thus each factor $W_n$ is also
a co-$H$ space. The following lemma gives a general criterion for
decomposing the retracts of graded co-$H$ spaces in term of grading
factors.

\begin{lem}~\cite{GTW}\label{lemma4.4}
Let $W$ be a simply connected $p$-local graded co-$H$ space of
finite type. Let $f\colon W\to W$ be a self-map such that on mod $p$
homology
\begin{enumerate}
\item[1)] $ f_*\colon \bar H_*(W)\to \bar H_*(W)$
preserves the filtration.
\item[2)] The induced bigraded map $E^0f_*$ is an idempotent:
$$E^0f_*\circ E^0f_*=E^0f_*\colon E^0H_*(W)\to E^0H_*(W).$$
\end{enumerate}
Let $A(f)=\hocolim_f W$ be the homotopy colimit and let
$A_n(f)=\hocolim_{g_n}W_n$, where $g_n$ is the composite
$$
g_n\colon W_n\rInto \bigvee_{k=1}^\infty W_k\rTo^{\phi_W^{-1}}
W\rTo^f W\rTo^{\phi_W} \bigvee_{k=1}^\infty W_k \rOnto W_n.
$$
Then there is a canonical homotopy decomposition of the homotopy
colimit
$$
A(f)\simeq \bigvee_{n=1}^\infty A_n(f)
$$
such that
$$
\bar H_*(A_n(f))\cong \im(E^0_nf_*\colon E^0H_*(W)\cong \bar
H_*(W_n)\to E^0H_*(W)\cong \bar H_*(W_n)).
$$
\hfill $\Box$
\end{lem}

Let $X$ be any path-connected space. Let $H_*(\Omega \Sigma X)$ be
filtered by the powers of the augmentation ideal filtration. From
the classical suspension splitting Theorem~\cite{James1}
$$
\phi\colon \Sigma \Omega\Sigma X\rTo^{\simeq}
\bigvee_{n=1}^\infty\Sigma X^{(n)},
$$
$\Sigma\Omega\Sigma X$ is a simply graded co-$H$ space and the
filtration
$$
I^t\bar H_*(\Sigma \Omega\Sigma X)=\phi_*^{-1}(\bar
H_*(\bigvee_{n=t}^\infty \Sigma X^{(n)}))
$$
coincide with (the suspension) of the augmentation ideal filtration
of $H_*(\Omega\Sigma X)$.

Let $W$ and $W'$ be graded spaces. Then $W\wedge W'$ is a graded
space with homotopy equivalence
$$
W\wedge W'\rTo^{\phi_W\wedge
\phi_{W'}}_{\simeq}(\bigvee_{n=1}^\infty W_n)\wedge
(\bigvee_{n=1}^\infty W'_n)=\bigvee_{n=1}^\infty \bigvee_{i=1}^{n-1}
W_i\wedge W'_{n-i}
$$
and the homology $\bar H_*(W\wedge W')=\bar H_*(W)\otimes \bar
H_*(W')$ is the tensor product of filtered modules. Moreover if $W$
is a graded co-$H$ space, then $W\wedge W'$ is a graded co-$H$
space.

Recall that~\cite{SW1} any natural coalgebra retract $A(V)$ of $T(V)$ for
ungraded modules admits the tensor length decomposition
$$
A(W)=\bigoplus_{n=0}^\infty A_n(W)
$$
for any graded or ungraded module $W$.

\begin{thm}[Suspension Splitting Theorem~\cite{GTW}]\label{theorem4.8}
Let $A(V)$ be any natural coalgebra retract of $T(V)$ for any
ungraded modules $V$ and let $\bar A$ be the geometric realization of
$A$.Then for any $p$-local simply connected co-$H$ space $Y$ of
finite type and any $p$-local path-connected co-$H$-space $Z$, there
is a functorial splitting
$$
Z\wedge \bar A(Y)\simeq \bigvee_{n=1}^\infty  [Z\wedge \bar A (Y)]_n
$$
such that
$$
\bar H_*([Z\wedge \bar A(Y)]_n)\cong \bar H_*(Z)\otimes
A_{n}(\Sigma^{-1}\bar H_*(Y))
$$
for each $n\geq 1$. In particular, for a $p$-local simply connected
co-$H$-space $Y$ there is a functorial suspension splitting
$$
\Sigma \bar A(Y)\simeq \bigvee_{n=1}^\infty \bar A_n(Y)
$$
such that
$$
\Sigma^{-1}\bar H_*(\bar A_n(Y))\cong A_n(\Sigma^{-1}\bar H_*(Y))
$$
for each $n\geq 1$.\hfill $\Box$
\end{thm}

\subsection{Hopf Invariants}
In this section, we review the results in Hopf invariants from~\cite{GTW}.
Let $A(V)$ be a natural coalgebra retract of $T(V)$ for ungraded
modules. From the suspension splitting theorem
(Theorem~\ref{theorem4.8}), there is a decomposition
$$
\Sigma \bar A(Y)\simeq \bigvee_{n=1}^\infty \bar A_n(Y)
$$
and so it induces Hopf invariants
$$
H_n\colon \bar A(Y)\rTo \Omega (\bar A_n(Y)).
$$
For computational purpose on homology, it is useful to make a
particular choice of Hopf invariants $H_n$. Let
$$
\calH\colon \Sigma\Omega\Sigma X\simeq \Sigma J(X)\rTo^{\simeq}
\bigvee_{n=1}^\infty \Sigma X^{(n)}
$$
be the fat James-Hopf invariants. Let $H_*(\Omega\Sigma X)$ be
filtered by the products of the augmentation ideal and let
$H_*(\bigvee_{n=1}^\infty X^{(n)})$ be filtered by
$$
\bigoplus_{t\geq n} H_*(X^{(t)}).
$$
By~\cite[Proposition 3.7]{SW2}, the isomorphism
$$
\calH_*\colon H_*(\Sigma\Omega\Sigma X)\rTo H_*(\bigvee_{n=1}^\infty \Sigma X^{(n)})
$$
preserves the filtration. Note that the composite
$$
\Sigma\Omega\Sigma X\rTo^{\calH} \bigvee_{n=1}^\infty \Sigma
X^{(n)}\rTo^{\mathrm{proj}} \Sigma X^{(n)}
$$
is the James-Hopf invariant $H_n$. Let $Y\in\CoH$ be a simply
connected $p$-local co-$H$-space of finite type with the
cross-section map $s_{\mu'}\colon Y\rTo \Sigma \Omega Y$. Let $f$ be
the composite
$$
\Sigma \Omega \Sigma \Omega Y\rTo^{\Sigma\Omega \sigma} \Sigma\Omega
Y\rTo^{\Sigma \Omega s_{\mu'}}\Sigma \Omega \Sigma \Omega Y.
$$
By Lemma~\ref{lemma4.4}, the space $[\Sigma\Omega Y]_n$ is defined
to be the homotopy colimit of the self map $g_n(Y)$ given by the
composite
$$
\begin{array}{ccccc}
\Sigma (\Omega Y)^{(n)} &\rInto& \bigvee_{k=1}^\infty \Sigma (\Omega
Y)^{(k)}&\rTo^{\calH^{-1}}&\Sigma \Omega \Sigma \Omega
Y\\
&\rTo^{\Sigma\Omega (s_{\mu'}\circ \sigma)} &\Sigma \Omega \Sigma
\Omega Y &\rTo^{\calH}& \bigvee_{k=1}^\infty \Sigma (\Omega
Y)^{(k)}\\
&\rOnto^{\mathrm{proj}}& \Sigma (\Omega Y)^{(n)}.&&\\
\end{array}
$$

\begin{prop}\label{GTW}
For any $Y\in \CoH$, the composite
$$
\begin{array}{ccccc}
\calH^Y\colon \Sigma\Omega Y &\rTo^{\Sigma\Omega s_{\mu'}}
&\Sigma\Omega \Sigma\Omega Y &\rTo^{\calH}_{\cong}&
\bigvee_{n=1}^\infty \Sigma (\Omega
Y)^{(n)}\\
&\rTo& \bigvee_{n=1}^\infty \hocolim_{g_n}\Sigma (\Omega Y)^{(n)}\\
\end{array}
$$
is a homotopy equivalence.\hfill $\Box$
\end{prop}
Now define the $n\,$th Hopf invariant $H_n^Y$ to be the adjoint map
to the composite
\begin{equation}\label{equation4.21}
\Sigma \Omega Y \rTo^{\calH^Y} \bigvee_{n=1}^\infty
\hocolim_{g_n}\Sigma (\Omega Y)^{(n)}\rTo^{\mathrm{proj}}
\hocolim_{g_n}\Sigma (\Omega Y)^{(n)}=[\Sigma\Omega Y]_n.
\end{equation}
\begin{thm}~\cite{GTW}\label{theorem4.11}
For any $Y\in\CoH$, there is a commutative diagram
\begin{diagram}
E^0H_*(\Omega Y)&\rTo^{E^0H^Y_{n\ast}}& E^0H_*(\Omega [\Sigma\Omega
Y]_n)\\
\uTo>{\cong}&&\uTo>{\cong}\\
T(\Sigma^{-1}\bar H_*(Y))&\rTo^{H_n}& T((\Sigma^{-1}\bar H_*(Y))^{\otimes n}),\\
\end{diagram}
where $H_n\colon T(V)\to T(V^{\otimes n})$ is the algebraic
James-Hopf map.\hfill $\Box$
\end{thm}

\section{Proof of Theorem~\ref{theorem1.5}}\label{section5}
\subsection{Proof of Theorem~\ref{theorem1.5}}
Now we give the proof of Theorem~\ref{theorem1.5}. All spaces are
localized at $p>2$. For a natural coalgebra retract $A$ of $T$, let
$\bar A$ be its geometric realization. According to the suspension
splitting theorem~\ref{theorem4.8}, there is a decomposition
$$
\Sigma \bar A(Y)\simeq \bigvee_{n=1}^\infty \bar A_n(Y).
$$
In particular, we have notations $\bar A^{\min}$ and $\bar
A^{\min}_n$. Let $X$ be a path-connected finite complex. Define $$
b_X=\sum\limits_{q=1}^\infty q\dim \bar H_q(X;\Z/p).$$

\begin{lem}\label{lemma5.1}
Let $Y$ be a simply connected co-$H$-space such that $\bar
H_{\odd}(Y)=0$ and $\dim \bar H_*(Y)=p-1$. Then
$$\bar
A^{\min}_{p^k-1}(Y)\simeq S^{(b_Y-p+1)\frac{p^k-1}{p-1}+1}.$$
\end{lem}
\begin{proof}
Let $V=\Sigma^{-1}\bar H_*(Y)$. Then $V_{\even}=0$ and $\dim V=p-1$.
By Theorem~\ref{theorem1.1},
$$
A^{\min}(V)=\bigotimes_{k=0}^\infty E(\bar L_p^k(V)).
$$
By considering tensor length $p^k-1$, we have
$$
A^{\min}_{p^k-1}(V)=E_{p-1}(V)\otimes E_{p-1}(\bar L_p(V))\otimes
\cdots\otimes E_{p-1}(\bar L_p^{k-1}(V)),
$$
which is a one-dimensional module. The assertion follows from
Theorem~\ref{theorem4.8}.
\end{proof}

Let $\bar Q^{\max}_n$ the geometric
realization of the indecomposables of $B^{\max}(V)$ as given in Theorem~\ref{theorem4.2}. Note that $Y$ is a retract of $\Sigma \bar
A^{\min}(Y)$ by Theorem~\ref{theorem4.8}.

\begin{lem}\label{lemma5.2}
Let $Y$ be any $p$-local simply connected co-$H$ space of finite
type. Then
\begin{enumerate}
\item there is a splitting cofibre sequence
$$
\bar Q^{\max}(Y)\longrightarrow \bar A^{\min}(Y)\wedge Y
\longrightarrow (\Sigma \bar A^{\min}(Y))/Y.
$$
Thus there is a decomposition $ Y\wedge \bar A^{\min}(Y) \simeq \bar
Q^{\max}(Y)\vee (\Sigma \bar A^{\min}(Y))/Y. $
\item There is a decomposition
$$
[Y\wedge \bar A^{\min}(Y)]_n\simeq \tilde Q^{\max}_n(Y)\vee \tilde
A^{\min}_n(Y).
$$
\item If $Y=\Sigma X$, then
$$
[Y\wedge \bar A^{\min}(Y)]_n\simeq X\wedge \bar A^{\min}_{n-1}(Y).
$$
\item If $\bar A^{\min}_{n-1}(Y)\simeq \Sigma Z$ for some $Z$,
then
$$
[Y\wedge \bar A^{\min}(Y)]_n\simeq Y\wedge Z.
$$
\item $\bar A^{\min}_p(Y)\simeq [Y\wedge \bar
A^{\min}(Y)]_p$.
\end{enumerate}
\end{lem}
\begin{proof}
(1) and (2). Since $\bar B^{\max}(Y)\simeq \Omega\bar Q^{\max}(Y),$
there is a fibre sequence
$$
\Omega Y\rTo^{\partial} \bar A^{\min}(Y)\rTo^{\simeq\ast} \bar Q^{\max}(Y)\rTo Y.
$$
There is a (right) action of $\Omega Y$ on $\bar A^{\min}(Y)$ such
that the diagram
\begin{diagram}
\Omega Y\times \Omega Y&\rTo&\Omega Y\\
\dTo>{\partial\times \id_{\Omega Y}}&&\dTo>{\partial}\\
\bar A^{\min}(Y)\times \Omega Y&\rTo& \bar A^{\min}(Y)\\
\end{diagram}
commutes. Consider the commutative diagram
\begin{diagram}
           & & & &\Sigma \Omega \bar Q^{\max}(Y) &\lInto& \bar
           Q^{\max}(Y)\\
           &&&&\dTo &&\\
(\Omega Y) \wedge Y &\rTo^{\id_{\Omega Y}\wedge s}& \Sigma\Omega
Y\wedge \Omega
Y&\rTo^{H}&\Sigma\Omega Y&\rTo& (\Sigma\Omega Y)/Y\\
\dTo&& \dTo &&\dTo&&\dTo\\
 \bar A^{\min}(Y)\wedge Y&\rTo^{\id_{\bar A^{\min}(Y)}\wedge s}&
\bar A^{\min}(Y)\wedge \Sigma \Omega Y &\rTo^{H}& \Sigma \bar
A^{\min}(Y)&\rTo& (\Sigma
\bar A^{\min}(Y))/Y,\\
\end{diagram}
where $H$ is the Hopf construction. Observe that the composite of
the middle row is a homotopy equivalence. It follows that there is a
sequence
\begin{equation}\label{equation5.1}
\bar Q^{\max}(Y)\rTo^f \bar A^{\min}(Y)\wedge Y\rTo^g (\Sigma \bar
A^{\min}(Y))/Y
\end{equation}
such that $g\circ f\simeq\ast$, where $g$ is the composite of the
bottom rows and $f$ is the composite of the maps from top row down
to the middle right composing with the homotopy inverse from
$(\Sigma\Omega Y)/Y$ to $\Omega Y \wedge Y$ and the left column map.
Let $C_f$ be the homotopy cofibre of the map $f$. By
Lemma~\ref{lemma3.1}, the resulting map
$$
C_f\longrightarrow (\Sigma \bar A^{\min}(Y))/Y
$$
is a homotopy equivalence. Thus Sequence~(\ref{equation5.1}) is a
cofibre sequence. Let $r$ be the composite
\begin{equation}\label{equation5.2}
r\colon \bar A^{\min}(Y)\wedge Y\rInto^{s^{\min}\wedge \id_Y}
(\Omega Y)\wedge Y \rTo^{H\circ (\id_{\Omega Y}\wedge s)} \Sigma
\Omega Y \rOnto \Sigma\Omega \tilde Q^{\max}(Y)\rTo^{\sigma} \bar
Q^{\max}(Y).
\end{equation}
By the proof of ~\cite[Lemma 2.37]{Wu}, on homology
$$
(r\circ f)_*\colon \bar H_*(\tilde Q^{\max}(Y))\rTo \bar H_*(\bar
Q^{\max}(Y))
$$
is an isomorphism. Thus $f\colon \bar Q^{\max}(Y)\to \bar
A^{\min}(Y)\wedge Y$ admits a retraction and hence assertion (1).
Let $H_*(\Omega Y)$ be filtered by the products of the augmentation
ideal. Then $H_*(\bar A^{\min}(Y))$ and $H_*(\bar Q^{\max}(Y)$ have
the induced filtration. Observe that the maps $f,g$ and $r$ induces
filtration preserving maps on homology. Thus the decomposition
$$
Y\wedge \bar A^{\min}(Y) \simeq \bar Q^{\max}(Y)\vee (\Sigma \bar
A^{\min}(Y))/Y
$$
induces a graded decomposition and hence assertion (2).

(3) By Theorem~\ref{theorem4.8}, $$\Sigma \bar A^{\min}(Y)\simeq
\bigvee_{k=1}^\infty \bar A^{\min}_k(Y).$$ Let $s\colon \bar
A^{\min}_{n-1}(Y)\rInto \Sigma\bar A^{\min}(Y)$ be the inclusion.
Then composite
$$
\begin{array}{lcl}
X\wedge \bar A^{\min}_{n-1}(Y)&\rInto^{\id_X\wedge s}& X\wedge
\Sigma \bar A^{\min}(Y)\\
&\simeq &Y\wedge \bar A^{\min}(Y)\\
&\simeq &\bigvee\limits_{k=2}^\infty[Y\wedge \bar
A^{\min}(Y)]_k\\
&\rOnto^{\proj.}& [Y\wedge \bar A^{\min}(Y)]_n\\
\end{array}
$$
induces an isomorphism on homology and hence assertion (3).

(4) Let $s\colon Y\to \Sigma\Omega Y$ be a cross-section to the
evaluation map $\sigma \colon \Sigma\Omega Y\to Y$. Let $f$ be the
idempotent
$$
\Sigma \Omega Y\wedge \bar A^{\min}(Y)\rTo^{\sigma\wedge \id_{\bar
A^{\min}(Y)}} Y\wedge \bar A^{\min}(Y)\rInto^{s\wedge \id_{\bar
A^{\min}(Y)}} \Sigma\Omega Y\wedge \bar A^{\min}(Y)
$$
and let $g_n$ be the composite
$$
\begin{array}{lcl}
\Omega Y \wedge \bar A^{\min}_{n-1}(Y)&\simeq &[\Sigma \Omega Y
\wedge \bar A^{\min}(Y)]_n\\
&\rInto& \Sigma \Omega Y\wedge \bar A^{\min}(Y)\\
&\rTo^f &\Sigma \Omega Y\wedge \bar A^{\min}(Y)\\
&\simeq &\Omega Y\wedge \Sigma \bar A^{\min}(Y)\\
&\rOnto^{\proj.}& \Omega Y \wedge \bar A^{\min}_{n-1}(Y),\\
\end{array}
$$
where the top homotopy equivalence follows from assertion (3) with
$X=\Omega Y$. According to Lemma~\ref{lemma4.4},
$$ [Y\wedge \bar A^{\min}(Y)]_n\simeq \hocolim_{g_n}\Omega Y \wedge
\tilde A^{\min}_{n-1}(Y)$$ with the canonical retraction
$$
r\colon \Omega Y\wedge \bar A^{\min}_{n-1}(Y)\longrightarrow
[Y\wedge \bar A^{\min}(Y)]_n.
$$
Now the composite
$$
\begin{array}{lcllcllcl}
Y\wedge Z&\rTo^{s\wedge\id_Z}&\Sigma\Omega Y \wedge Z &\simeq
&\Omega Y\wedge \bar A^{\min}_{n-1}(Y)
&\rTo^r &[Y\wedge \bar A^{\min}(Y)]_n\\
\end{array}
$$
induces an isomorphism on homology and so it is a homotopy
equivalence. Assertion (4) follows.

(5). According to~\cite[Section 11.2, p.97]{SW1}, $Q^{\max}_p(V)=0$.
Thus $$\bar H_*(\bar Q^{\max}_p(Y))=0$$ and so $\bar
Q^{\max}_p(Y)\simeq\ast$. By assertion (2),
$$\bar A^{\min}_p(Y)\simeq [Y\wedge\bar A^{\min}(Y)]_p$$ and hence
the result.
\end{proof}

\begin{lem}\label{lemma5.3}
Let $Y$ be a simply connected co-$H$-space such that $\bar
H_{\odd}(Y)=0$ and $\dim \bar H_*(Y)=p-1$. Then
$$
\bar A^{\min}_p(Y)\simeq \Sigma^{b_Y-p+1} Y.
$$
\end{lem}
\begin{proof}
By Lemma~\ref{lemma5.1}, $\bar A^{\min}_{p-1}=S^{b_Y-p+2}$. By
Lemma~\ref{lemma5.2} (4),
$$
[Y\wedge \bar A^{\min}(Y)]_p\simeq \Sigma^{b_Y-p+1}Y.
$$
The result then follows from Lemma~\ref{lemma5.2} (5).
\end{proof}

\begin{proof}[Proof of Theorem~\ref{theorem1.5}]
(1) and (2). Let
$$
H^Y_p\colon \Omega Y\rTo \Omega [\Sigma\Omega Y]_p
$$
be the James-Hopf invariant. By Lemma~\ref{lemma5.3},
$$\bar
A^{\min}_p(Y)\simeq \Sigma^{b_Y-p+1}Y.
$$
Define the map $H_p$ by
the composite
$$
\begin{array}{ccccc}
\bar A^{\min}(Y)&\rInto& \Omega Y&\rTo^{H^Y_p}& \Omega( [\Sigma
\Omega Y]_p)\\
&\rOnto& \Omega (\bar A^{\min}_p(Y))&\rOnto& \bar A^{\min}(\bar
A^{\min}_p(Y))=\bar A^{\min}(\Sigma^{b_Y-p+1}Y).\\
\end{array}
$$
Let $\bar E(Y)$ be the homotopy fibre of the map $H_p$. Then there
is a fibre sequence
$$
\bar E(Y)\rTo \bar A^{\min}(Y)\rTo^{H_p} \bar A^{\min}(\bar
A^{\min}_p(Y)).
$$

We compute the homology of $\bar E(Y)$. Let $V=\Sigma^{-1}\bar
H_*(Y)$. According to Theorem~\ref{theorem4.11},
$$
E^0H_{p\ast}=H_p\colon T(V)=E^0H_*(\Omega Y)\rTo T(V^{\otimes
p})=E^0H_*(\Omega([\Sigma\Omega Y]_p))
$$
is the usual algebraic James-Hopf map. Thus there is a commutative
diagram
\begin{diagram}
T(V)&\rTo^{H_p}&T(V^{\otimes p})\\
\dOnto&&\dOnto\\
E^0H_*(\bar A^{\min}(Y))=A^{\min}(V)&\rTo^{E^0H_{p\ast}}&
E^0H_*(\bar
A^{\min}(\bar A^{\min}_p(Y)))\\
\end{diagram}
where $H_p$ is the algebraic James-Hopf map. By
Theorem~\ref{theorem1.1} and Lemma~\ref{lemma3.2}, the primitives of
$A^{\min}(V)$ is given by
$$
PA^{\min}(V)=\bigoplus_{k=0}^\infty \bar L^k_p(V).
$$
Let $\bar L$ be the sub Lie algebra of $L(V)$ generated by $L_p(V)$.
Note that
$$
\bar L_p^k(V)\subseteq \bar L
$$
for $k\geq 1$. By~\cite[Theorem 1.1]{Wu1}, the James-Hopf map
$$H_p|_{\bar L}\colon \bar L\rTo PT(V^{\otimes p})$$
is a morphism of Lie algebras. Since
$$H_{p}|_{L_p}\colon \bar L_p\to
V^{\otimes p}$$ is the canonical inclusion,
$$
H_p(\bar L_p^{k+1}(V))=\bar L_p^k(\bar L_p(V))
$$
for each $k\geq 0$. Note that $\bar L_p(V)\subseteq A^{\min}_p(V)$
and, since $\dim V=p-1$,
$$
\dim \bar L_p(V)=\dim A^{\min}_p(V) =p-1.
$$
Thus $\bar L_p(V)=A^{\min}_p(V)$. It follows that
\begin{equation}\label{equation5.3}
PE^0H_{p\ast}\colon PE^0H_*(\bar A^{\min}(Y))\rTo PE^0H_*(\bar
A^{\min}(\bar A^{\min}_p(Y)))
\end{equation}
is onto with the kernel $L^0_p(V)=V$.

For any graded Hopf algebra $A$, let $A^{\ab}$ be the abelianization
of $A$ and let
$$
\phi\colon A\rTo A^{\ab}
$$
be the quotient map. Define $\bar\phi$ to be the composite
$$
H_*(\bar A^{\min}(Y))\rInto H_*(\Omega Y)\rTo^{\phi} H_*(\Omega
Y)^{\ab}.
$$
Since $H_*(\Omega Y)$ is generated by odd dimensional elements,
$H_*(\Omega Y)^{\ab}$ is an exterior algebra. Moreover
$E^0H_*(\Omega Y)^{\ab}$ is primitively generated exterior algebra.
Thus
\begin{equation}\label{equation5.4}
E^0H_*(\Omega Y)^{\ab}=E(V).
\end{equation}
Consider the map
\begin{equation}\label{equation5.5}
\begin{array}{ccc}
\theta\colon H_*(\bar A^{\min}(Y))&\rTo^{\psi}& H_*(\bar
A^{\min}(Y))\otimes H_*(\bar A^{\min}(Y))\\
&\rTo^{\bar\phi\otimes H_{p\ast}}& H_*(\Omega Y)^{\ab}\otimes
H_*(\bar A^{\min}(\bar A^{\min}_p(Y))).\\
\end{array}
\end{equation}
From equations~(\ref{equation5.3}) and ~(\ref{equation5.4}),
$$
E^0P\theta\colon PE^0H_*(\bar A^{\min}(Y))\rTo PE^0(H_*(\Omega
Y)^{\ab}\otimes H_*(\bar A^{\min}(\bar A^{\min}_p(Y))))
$$
is an isomorphism. Thus
$$
E^0\theta\colon E^0H_*(\bar A^{\min}(Y))\rTo E^0(H_*(\Omega
Y)^{\ab}\otimes H_*(\bar A^{\min}(\bar A^{\min}_p(Y))))
$$
is a monomorphism. Since both sides have the same Poincar\'e series,
$E^0\theta$ is an isomorphism and so $\theta$ is an isomorphism.
Thus there is an algebra isomorphism on cohomology
$$
H^*(\bar A^{\min}(Y))\cong H^*(\bar A^{\min}(\bar
A^{\min}_p(Y)))\otimes (H_*(\Omega Y)^{\ab})^*.
$$
By the Eilenberg-Moore spectral sequence, there is an epimorphism
$$
H_*(\Omega Y)^{\ab}\rOnto H_*(\bar E(Y))
$$
and so the Poincar\'{e} series
$$
\chi(H_*(\bar E(Y)))\leq \frac{\chi(H_*(\bar
A^{\min}(Y)))}{\chi(H_*(\bar A^{\min}(\bar A^{\min}_p(Y))))}.
$$
By applying the Serre spectral sequence to the fibre sequence
$$
\bar E(Y)\rTo \bar A^{\min}(Y)\rTo \bar A^{\min}(\bar
A^{\min}_p(Y)),
$$
we have the $E^2$-terms given by $$H_*(\bar E(Y))\otimes H_*(\bar
A^{\min}(\bar A^{\min}_p(Y)))$$ and so
$$
\chi(H_*(\bar A^{\min}(Y)))\leq \chi(H_*(\bar E(Y))\chi(H_*(\bar
A^{\min}(\bar A^{\min}_p(Y)))).
$$
It follows that
$$
\chi(H_*(\bar E(Y)))=\frac{\chi(H_*(\bar
A^{\min}(Y)))}{\chi(H_*(\bar A^{\min}(\bar A^{\min}_p(Y))))}
$$
and so
$$
H_*(\Omega Y)^{\ab}\cong H_*(\bar E(Y)).
$$
This finishes the proof of assertion (1) and (2).

(3). Let $\bar Y$ be the homotopy cofibre of the map $f\colon S^n\to
Y$. Since $f$ is a co-$H$-map, $\bar Y$ is a co-$H$-space. By the
naturality of $\bar A^{\min}$, the pinch map $q\colon Y\to \bar Y$
induces a map
$$
\bar A^{\min}(Y)\rTo \bar A^{\min}(\bar Y).
$$
Let $V'=\Sigma^{-1}\im(f_*\colon \bar H_*(S^n)\to \bar H_*(Y))$ and
let $\bar V=\Sigma^{-1}\bar H_*(\bar Y)\cong V/V'$. Since $\dim \bar
V=p-2$,
\begin{equation}\label{equation5.6}
E^0H_*(\bar A^{\min}(\bar Y))=E(\bar V)
\end{equation}
by~\cite[Corollary 11.6]{SW1}. From the proof of assertions (1) and
(2), the composite
\begin{equation}\label{equation5.7}
H_*(\bar E(Y))\rInto H_*(\bar A^{\min}(Y))\rInto H_*(\Omega Y)\rOnto
H_*(\Omega Y)^{\ab}
\end{equation}
is an isomorphism.  Let $g$ be the composite
$$
\bar E(Y)\rTo \bar A^{\min}(Y)\rTo \bar A^{\min}(\bar Y).
$$
Consider the commutative diagram
\begin{equation}\label{equation5.8}
\begin{diagram}
H_*(\bar E(Y))&\rInto& H_*(\bar A^{\min}(Y))& \rInto &H_*(\Omega Y)&
\rOnto& H_*(\Omega Y)^{\ab}\\
&\rdTo>{g_*}&\dTo&&\dOnto&&\dOnto\\
&  & H_*(\bar A^{\min}(\bar Y))&\rInto& H_*(\Omega \bar Y)&\rOnto & H_*(\Omega\bar Y)^{\ab},\\
\end{diagram}
\end{equation}
where the composites in the top and bottom rows are isomorphisms.
Thus $$g_*\colon H_*(\bar E(Y))\to H_*(\bar A^{\min}(Y))$$ is onto.
By~\cite[Proposition 4.20]{MM}, both $H_*(\Omega Y)^{\ab}$ and
$H_*(\Omega \bar Y)^{\ab}$ are primitively generated because they
are commutative with trivial restricted maps. Thus
$$
H_*(\Omega \bar Y)^{\ab}\cong E(\bar V) \textrm{ and } H_*(\Omega
Y)^{\ab}\cong E(V)\cong E(\bar V)\otimes E(V')
$$
as Hopf algebras. From diagram~\ref{equation5.8} the map $g_*$
induces a coalgebra decomposition
$$
H_*(\bar E(Y))\cong H_*(\bar A^{\min}(\bar Y))\otimes E(V').
$$
It follows that the homotopy fibre of the map $g\colon \bar E(Y)\to
\bar A^{\min}(\bar Y)$ is $S^{n-1}$ by using the Eilenberg-Moore
spectral sequence. Thus there is a homotopy commutative diagram of
fibre sequences
\begin{diagram}
\Omega\bar A^{\min}(\bar A^{\min}_p(Y))&\rTo^{P_f}& S^{n-1}&\rTo&
B_f&
\rTo&\bar A^{\min}(\bar A^{\min}_p(Y))\\
\dEq&&\dTo&&\dTo&&\dEq\\
\Omega\bar A^{\min}(\bar A^{\min}_p(Y))&\rTo^{P}& \bar
E(Y)&\rTo^{E}& \bar A^{\min}(Y)&
\rTo^{H_p} &\bar A^{\min}(\bar A^{\min}_p(Y))\\
 & &\dTo&&\dTo&&\\
 & &\bar A^{\min}(\bar Y)&\rEq&\bar A^{\min}(\bar Y)& &\\
\end{diagram}
and hence the result.
\end{proof}

\subsection{Self-maps of $\bar A^{\min}(Y)$}
Let $Y$ be a simply connected co-$H$-space of $(p-1)$-cell complex
with the cells in even dimensions. In this case, the integral
homology of $Y$ is torsion free and so we can take homology $H_*(Y)$
over $p$-local integers. Let $f$ be any self-map of $\bar
A^{\min}(Y)$. Let $V=\Sigma^{-1}\bar H_*(Y)$ with a basis
$\{x_{l_1},\ldots,x_{l_{p-1}}\}$, where the degree $|x_{l_i}|=l_i$
with $l_1\leq\cdots \leq l_{p-1}$. By Theorem~\ref{theorem1.5},
there is a coalgebra isomorphism
$$
H_*(\bar A^{\min}(Y))\cong \bigotimes_{k=0}^\infty H_*(\bar E(\bar
A^{\min}_{p^k}(Y)))\cong \bigotimes_{k=0}^\infty
E(A^{\min}_{p^k}(V))\cong A^{\min}(V).
$$
Thus the primitives
$$
PH_*(\bar A^{\min}(Y))=\bigoplus_{k=0}^\infty A^{\min}_{p^k}(V)
$$
and so on cohomology
$$
QH^*(\bar A^{\min}(Y))=\bigoplus_{k=0}^\infty A^{\min}_{p^k}(V)^*.
$$
Note that $A^{\min}_{p^k}(V)$ has a basis in the form
$$
w_{p^k,i}=[(x_{l_1}\wedge\cdots\wedge
x_{l_{p-1}})^{\frac{p^k-1}{p-1}}x_{l_i}]
$$
for $1\leq i\leq p-1$. Let $f^*_{k,l}$ be the composite
$$
A^{\min}_{p^k}(V)^*\rInto QH^*(\bar A^{\min}(Y))\rTo^{f^*} QH^*(\bar
A^{\min}(Y))\rTo^{\mathrm{proj.}} A^{\min}_{p^l}(V)^*.
$$
Note that $|w_{p^k,p-1}|<|w_{p^{k'},1}|$ for $k<k'$. Thus
\begin{equation}\label{equation5.9}
f^*_{k,l}=0 \textrm{ for } k\not=l
\end{equation}
and so
\begin{equation}\label{equation5.10}
f^*(A^{\min}_{p^k}(V)^*)\subseteq A^{\min}_{p^k}(V)^*.
\end{equation}
It follows that
$$
f^*\colon A^{\min}(V)^*=H^*(\bar A^{\min}(Y))\rTo H^*(\bar
A^{\min}(Y))=A^{\min}(V)^*
$$
is a graded map with respect to the tensor length. Since $\dim
A^{\min}_{p^k-1}(V)=1$, let $\deg^k(f)$ be the degree of the map
$$
f^*\colon A^{\min}_{p^k-1}(V)^*\rTo A^{\min}_{p^k-1}(V)^*.
$$
\begin{prop}
Let $f\colon \bar A^{\min}(Y)\to \bar A^{\min}(Y)$ be any map. Then
$\deg^k(f)$ over $p$-local integers is invertible if and only if
$$
f_*\colon H_j(\bar A^{\min}(Y);\Z/p)\rTo H_j(\bar A^{\min}(Y);\Z/p)
$$
is an isomorphism for $j< (l_1+\cdots +l_{p-1})\cdot
\frac{p^k-1}{p-1}+l_1$.
\end{prop}
\begin{proof}
The assertion follows from that
$$
\deg^k(f)=\prod_{j=0}^{k-1} \det\left(f^*\colon
A^{\min}_{p^j}(V)^*\rTo A^{\min}_{p^j}(V)^*\right)
$$
for cohomology with coefficients over $p$-local integers.
\end{proof}

Since $\bar A^{\min}(Y)$ is an $H$-space, there is a product $f\ast
g$ for any self maps $f$ and $g$ of $\bar A^{\min}(Y)$.

\begin{prop}
Let $f$ and $g$ be any self maps of $\bar A^{\min}(Y)$. Then
$$
\deg^k(f\ast g)=\prod_{j=0}^{k-1}\det\left(f_*+g_*\colon
A^{\min}_{p^j}(V)\rTo A^{\min}_{p^j}(V)\right).
$$
\end{prop}
\begin{proof}
Note that
$$
(f\ast g)_*(w)=f_*(w)+g_*(w)
$$
for any primitive element $w$. The assertion follows.
\end{proof}

\section{Proof of Theorem~\ref{theorem1.6} }\label{section6}
\begin{lem}\label{lemma6.1}
Let $Y$ be a $p$-local simply connected co-$H$-space such that
$H_{\odd}(Y)=0$ and $\dim \bar H_*(Y)=p-1$.
\begin{enumerate}
\item If $Y$ admits a nontrivial decomposition, then the $\EHP$
fibration splits off.
\item  If $Y$ is atomic, then $\bar E(Y)$ is also atomic.
\end{enumerate}
\end{lem}
\begin{proof}
(1). Suppose that $Y\simeq Y_1\vee Y_2$ be a nontrivial
decomposition. Since $\dim \bar H_*(Y_i)<p-1$, $H_*(\bar
A^{\min}(Y_i))=E(\Sigma^{-1}\bar H_*(Y_i))$. Let $p_i$ be the
composite
$$
\bar A^{\min}(Y)\rInto \Omega Y\rOnto^{\mathrm{proj.}}\Omega
Y_i\rOnto \bar A^{\min}(Y_i).
$$
Then the composite
$$
\bar E(Y)\rTo^E \bar A^{\min}(Y)\rTo^{(p_1,p_2)} \bar
A^{\min}(Y_1)\times \bar A^{\min}(Y_2)
$$
is a homotopy equivalence as it induces an isomorphism on homology.
Assertion (1) follows.

(2). Let $f\colon \bar E(Y)\to \bar E(Y)$ be any self map inducing an
isomorphism on the bottom cell.  Let $V=\Sigma^{-1}\bar H_*(Y)$. By
Theorem~\ref{theorem1.5}, $H_*(\bar E(Y))=E(V)$ as coalgebras. Then
$$
Pf_*\colon V=PE(V)\rTo V=PE(V)
$$
is an isomorphism on the bottom cell. From the suspension splitting
theorem (Theorem~\ref{theorem4.8}),
$$
\Sigma \bar E(Y)\simeq \bigvee_{n=1}^{p-1}\bar A^{\min}_n(Y)
$$
with $\Sigma^{-1} H_*(\bar A^{\min}_n(Y)=A^{\min}(V)$. Now the
composite
$$
g\colon Y=\bar A^{\min}_1(Y)\rInto \Sigma\bar E(Y)\rTo^{\Sigma f}
\Sigma \bar E(Y)\simeq \bigvee_{n=1}^{p-1}\bar A^{\min}_n(Y)\rOnto
\bar A^{\min}_1(Y)=Y
$$
is a homotopy equivalence because on homology
$\Sigma^{-1}g_*=Pf_*\colon V\to V$ is an isomorphism on the bottom
cell. It follows that $Pf_*\colon V\to V$ is an isomorphism and so
$f_*\colon E(V)\to E(V)$ is an isomorphism. Thus $f$ is a homotopy
equivalence and hence the result.
\end{proof}

Let $H_*(\Omega Y)$ be filtered by the products of the augmentation
ideal. For spaces $X$ and $Z$ such that $H_*(X)$ and $H_*(Z)$ are
filtered, a map $f\colon X\to Z$ is called \textit{filtered} of
$f_*\colon H_*(X)\to H_*(Z)$ is a filtered map.  A retract $X$ of
$\Omega Y$ is called \textit{filtered} if there is a filtered self
map $f$ of $\Omega Y$ such that $X\simeq \hocolim_f\Omega Y$. Note
that for each filtered retract $X$ of $\Omega Y$ there is a
filtration on $H_*(X)$ induced from the filtration of $H_*(\Omega
Y)$. Consider all possible filtered retracts of $\Omega Y$. Divide
them in two types: a filtered retract $X$ of $\Omega Y$ is called of
\textit{type $A$} if $X$ contains the bottom cells of $\Omega Y$,
otherwise it is called of \textit{type B}.

\begin{lem}\label{lemma6.2}
For any $p$-local simply connected co-$H$-space $Y$, there exists a
minimal filtered retract $X^{\min}(Y)$ of $\Omega Y$ of type $A$
such that: (1) $X^{\min}(Y)$ is a filtered retract of $\Omega Y$ of
type $A$ and (2) any filtered retract of $\Omega Y$ of type $A$ is a
filtered retract of $X(Y)$. Moreover there is a filtered
decomposition
$$
\Omega Y\simeq X^{\min}(Y)\times \Omega Q.
$$
\end{lem}
\begin{proof}
Let $\calS$ be the set of all filtered retract of $\Omega Y$ of type
$A$. Define a partial order on $\calS$ by setting $X\leq X'$ if $X$
is a filtered retract of $X'$. From $X_1,X_2\in \calS$, let $f$ be
the composite
$$
\Omega Y\rOnto X_1\rInto \Omega Y\rOnto X_2\rInto \Omega Y.
$$
Let $X_3=\hocolim_f\Omega Y$. Then $X_3$ is common filtered retract of
$X_1$ and $X_2$. Namely $X_3\leq X_1$ and $X_3\leq X_2$. By Zorn
Lemma, $\calS$ has the minimal element $X^{\min}(Y)$. Consider the
commutative diagram
\begin{diagram}
\Omega\Sigma X^{\min}(Y)&\rTo& X^{\min}(Y)&\rTo &\Sigma
X^{\min}(Y)\wedge X^{\min}(Y)&\rTo^{H}&\Sigma X^{\min}(Y)\\
\uTo>{\Omega j}&&\uEq&&\uTo&\textrm{pull}& \uTo>{j}\\
\Omega Y&\rTo^{\partial}& X^{\min}(Y)&\rTo& Q&\rTo& Y\\
\end{diagram}
where $j$ is the adjoint map of the retraction $\Omega Y\to
X^{\min}(Y)$. Then the composite
$$
X^{\min}(Y)\rTo \Omega Y\rTo^{\partial} X^{\min}(Y)
$$
is a homotopy equivalence by the minimal assumption of
$X^{\min}(Y)$. It follows that there is a filtered decomposition
$$
\Omega Y\simeq X^{\min}(Y)\times \Omega Q
$$
and hence the result.
\end{proof}

\begin{proof}[Proof of Theorem~\ref{theorem1.6}]
\noindent $(1)\Longrightarrow (2).$ If $\EHP$ fibration splits off,
then $\bar E(Y)$ is a retract of $\bar A^{\min}(Y)$.  Note that
$\bar A^{\min}(Y)$ is an $H$-space as it is a retract of $\Omega Y$.
Thus $\bar E(Y)$ is an $H$-space.

$(2)\Longrightarrow (3)$. From the suspension splitting theorem
(Theorem~\ref{theorem4.8}),
$$
\Sigma \bar A^{\min}(Y)\simeq \bigvee_{n=1}^\infty \bar
A^{\min}_n(Y)
$$
with $\bar A^{\min}_1(Y)=Y$ and $\Sigma^{-1}\bar H_*(\bar
A^{\min}_n(Y))=A^{\min}(\Sigma^{-1}\bar H_*(Y))$. The composite
$$
\Sigma \bar E(Y)\rTo \Sigma \bar A^{\min}(Y)\rTo^{\mathrm{proj.}}
\bigvee_{n=1}^{p-1}\bar A^{\min}_n(Y)
$$
is a homotopy equivalence because $H_*(\bar E(Y))\to H_*(\bar
A^{\min}(Y))$ is a monomorphism and the connectivity of
$\bigvee_{n=p}^\infty \bar A^{\min}_n(Y)$ is greater than the
dimension of $\Sigma \bar E(Y)$. Thus $Y$ is a retract of
$\Sigma\bar E(Y)$. By Theorem~\ref{theorem1.5}, $H^*(\bar
E(Y))=E(\Sigma^{-1}\bar H^*(Y))$ and so $\bar E(Y)$ is an $H$-space
with $Y$ as a retractile generating complex.

 $(3)\Longrightarrow (1)$.
Let $X$ be any $H$-space having $Y$ as a retractile generating
complex. Since $Y$ is a retract of $X$. Let $s\colon Y\to \Sigma X$
be an inclusion and let $r\colon \Sigma X\to Y$ be the retraction
such that $H^*(X)$ is generated by $M=r^*(\Sigma^{-1}\bar H_*(Y))$
with $M\cong QH^*(X)$. Let $V=\Sigma^{-1}\bar H_*(Y)$. Since
$H_{\odd}(Y)=0$, $V_{\even}=0$. Note that $H_*(X)$ is a quasi-Hopf
algebra as $X$ is an $H$-space. Thus $H^*(X)$ is the exterior
algebra generated by $V^*$ by the Borel Theorem~\cite[Theorem
7.11]{MM}. There is a homotopy commutative diagram of fibre
sequences
\begin{equation}\label{equation6.1}
\begin{diagram}
\Omega\Sigma X&\rTo^{\tilde\partial}& X&\rTo&\Sigma X\wedge X&\rTo^{H}& \Sigma X\\
\uTo>{\Omega s}&&\uEq&&\uTo&\textrm{pull}&\uTo>{s}\\
\Omega Y&\rTo^{\partial}& X&\rTo& B&\rTo& Y,\\
\end{diagram}
\end{equation}
where $H\colon \Sigma X\wedge X\to \Sigma X$ is the Hopf fibration
for the $H$-space $X$. Let $r'\colon X\to \Omega Y$ be the adjoint
map of $r$. Since the dimension of $X$ less than the connectivity of
$\bar A^{\min}(\Sigma^{b_Y-p+1}(Y))$, the composite
$$
X\rTo^{r'}\Omega Y\rOnto \bar A^{\min}(Y)
$$
lifts to the fibre $\bar E(Y)$ of the $\EHP$-fibration. Let
$r''\colon X\to \bar E(Y)$ be a lifting of the above composite.
Consider the composite
\begin{equation}\label{equation6.2}
\begin{diagram}
\phi\colon \Omega Y&\rTo^{\partial}& X&\rTo^{r''}&\bar E(Y)&\rTo^E&\bar A^{\min}(Y)&\rInto& \Omega Y\\
\uInto>{j}&&\uInto>{j}&&\uInto>{j}& &\uInto>{j}&&\uInto>{j}\\
\bigvee_{\alpha}S^n&\rEq&\bigvee_{\alpha}S^n&\rEq&
\bigvee_{\alpha}S^n&\rEq&\bigvee_{\alpha}S^n&\rEq&\bigvee_{\alpha}S^n,\\
\end{diagram}
\end{equation}
where $j$ is the inclusion of the bottom cells. Let
$$
\bar X=\hocolim_{\phi}\Omega Y.
$$
Then $\bar X$ is a common retract of $X$, $\bar E(Y)$, $\bar
A^{\min}(Y)$ and $\Omega Y$, which contains the bottom cell. Since
$Y$ is atomic, $\bar E(Y)$ is also atomic by Lemma~\ref{lemma6.1}.
Thus the retraction
$$
\bar E(Y)\rTo^E \bar A^{\min}(Y)\rInto \Omega Y\rTo
\hocolim_{\phi}\Omega Y=\bar X
$$
is a homotopy equivalence. It follows that the $\EHP$-sequence
$$
\bar E(Y)\rTo^E \bar A^{\min}(Y)\rTo^{H_p} \bar
A^{\min}(\Sigma^{b_Y-p+1}Y)
$$
admits a retraction $\bar A^{\min}(Y)\to \bar E(Y)$ and so the
$\EHP$ fibration splits off.

Now $(1)\Longrightarrow (4)$ and $(4)\Longrightarrow (5)$ are
obvious. $(5)\Longrightarrow (6)$ follows from the homotopy
commutative diagram
\begin{diagram}
\Omega \bar A^{min}(\Sigma^{b_Y-p+1}Y)&\rTo^P& \bar E(Y)&\rTo^E&
\bar A^{\min}(Y)&\rTo^{H_p}& \bar A^{\min}(\Sigma^{b_Y-p+1}Y)\\
\uInto&&\uTo&&\uTo>{g}&&\uInto\\
\Sigma^{b_Y-p-1}Y&\rTo& \ast&\rTo&\Sigma^{b_Y-p}Y
&\rEq&\Sigma^{b_Y-p}Y,\\
\end{diagram}
where the bottom row is the cofibre sequence.

$(6)\Longrightarrow (3)$. Consider the commutative diagram
\begin{diagram}
\bar A^{\min}(\Sigma^{b_Y-p+1}Y)&\rInto& \Omega \Sigma^{b_Y-p+1}Y
&\rTo^{\tilde g}&\Omega Y&\rOnto& \bar
A^{\min}(Y)\\
 & \luInto&\uInto& & \uInto&\ruEq &\\
 &     &\Sigma^{b_Y-p}Y&\rTo^{g}&\bar A^{\min}(Y),& &\\
\end{diagram}
where $\tilde g$ is the $H$-map induced by $g$. Let
$\bigvee_{\alpha}S^q$ be the bottom cells of $\Sigma^{b_Y-p}Y$.
Observe that
\begin{equation}\label{equation6.3}
 H_{p\ast}\colon
H_q(\bar A^{\min}(Y))\rTo^{\cong} H_q(\bar
A^{\min}(\Sigma^{b_Y-p+1}Y)).
\end{equation}
Thus there is a commutative diagram
\begin{equation}\label{equation6.4}
\begin{diagram}
\theta\colon \Omega Y & \rOnto & \bar A^{\min}(Y)&\rTo^{H_p}& \bar
A^{\min}(\Sigma^{b_Y-p+1}Y)&\rInto& \Omega \Sigma^{b_Y-p+1}Y
&\rTo^{\tilde g}&\Omega Y\\
\uTo& &\uTo>{g|_{\vee_{\alpha}S^q}}& &\uInto>{j} & &\uInto>{j}& & \uTo\\
\bigvee_{\alpha}S^q & \rEq &\bigvee_{\alpha}S^q&\rEq &
\bigvee_{\alpha}S^q& \rEq& \bigvee_{\alpha}S^q&\rEq& \bigvee_{\alpha}S^q,\\
\end{diagram}
\end{equation}
where $j$ is the inclusion of the bottom cells.

Let $H_*(\Omega Y)$ be filtered by the products of the augmentation
ideal. Then $H_*(\bar A(Y))$ has the induced filtration. Observe
that
$$
g_*(\bar H_*(\Sigma^{b_Y-p}Y))\subseteq I^pH_*(\bar
A^{\min}(Y))\subseteq I^pH_*(\Omega Y)
$$
because the top dimension of $H_*(\bar
E(Y))=\oplus_{n=1}^{p-1}A^{\min}(\Sigma^{-1}Y)$ is less than the
connectivity of $\Sigma^{b_Y-p}Y$. It follows that
$$
\tilde g_*\colon
H_*(\Omega\Sigma^{b_Y-p+1}Y)=H_*(J(\Sigma^{b_Y-p}Y))\rTo H_*(\Omega
Y)
$$
preserves the filtration. Thus $\theta$ is a
filtered map because the other factors of $\theta$ are
filtration-preserving maps on homology. Let
$$
Z=\hocolim_{\theta} \Omega Y
$$
be the homotopy colimit. Then $Z$ is a common filtered retract of
$\Omega Y$, $\bar A^{\min}(Y)$, $\bar A^{\min}(\Sigma^{b_Y-p+1}Y)$,
and $\Omega \Sigma^{b_Y-p+1}Y$. By Lemma~\ref{lemma4.4}, $\Sigma Z$
is a common graded retract of $\Sigma \Omega Y$, $\Sigma \bar
A^{\min}(Y)$, $\Sigma \bar A^{\min}(\Sigma^{b_Y-p+1}Y)$, and
$\Sigma\Omega \Sigma^{b_Y-p+1}Y$. Thus $[\Sigma Z]_p$ is a retract
of
$$
[\Sigma \bar A^{\min}(\Sigma^{b_Y-p+1}Y)]_p=\Sigma^{b_Y-p+1}Y
$$
containing the bottom cells by diagram~(\ref{equation6.4}). It
follows that
\begin{equation}\label{equation6.5}
 [\Sigma Z]_p=[\Sigma \bar A^{\min}(\Sigma^{b_Y-p+1}Y)]_p=\Sigma^{b_Y-p+1}Y
 \end{equation}
 because
$\Sigma^{b_Y-p+1}Y$ is atomic. Thus $Z$ contains $\Sigma^{b_Y-p}Y$
as the bottom piece.

Consider the filtered decomposition
$$
\Omega Y\simeq \bar B^{\max}(Y)\times \bar A^{\min}(Y).
$$
Since $Z$ is a filtered retract of $\bar A^{\min}(Y)$, there is a
further decomposition
$$
\Omega Y\simeq \bar B^{\max}(Y)\times Z\times A',
$$
where $A'$ is a filtered retract of $\Omega Y$ of type $A$. By
Lemma~\ref{lemma6.2}, $A'$ admits a further filtered decomposition
$$
A'\simeq B'\times X^{\min}(Y).
$$
This gives a filtered decomposition
\begin{equation}\label{equation6.6}
\Omega Y\simeq \bar B^{\max}(Y)\times Z\times B'\times X^{\min}(Y)
\end{equation}
and so the space $Q$ given in Lemma~\ref{lemma6.2} has the property that
\begin{equation}\label{equation6.7}
\Omega Q\simeq \bar B^{\max}(Y)\times Z\times B'.
\end{equation}
On homology
\begin{equation}\label{equation6.8}
PE^0H_*(Z)\oplus PE^0H_*(\bar B^{\max}(Y))\subseteq PE^0H_*(\Omega
Q).
\end{equation}
Let $V=\Sigma^{-1}\bar H_*(Y)$. From equation~(\ref{equation6.5}),
\begin{equation}\label{equation6.9}
\bar L_p(V)=A^{\min}_p(V)=\Sigma^{-1}\bar H_*([\Sigma \bar
A^{\min}(\Sigma^{b_Y-p+1}Y)]_p)\subseteq PE^0(\Omega Q).
\end{equation}
By equation~(\ref{equation6.8}),
$$
\bigoplus_{n=2}^{p-1}Q^{\max}_n(V)\oplus \bar L_p(V)\subseteq
PH_*(\Omega Q).
$$
Since $H_*(\Omega Q)$ is a Hopf algebra, the sub Hopf algebra
\begin{equation}\label{equation6.10}
B=\la \bigoplus_{n=2}^{p-1}Q^{\max}_n(V)\oplus \bar L_p(V)\ra
\subseteq H_*(\Omega Q).
\end{equation}
By Lemma~\ref{lemma3.2}, $B=B^{[1]}(V)$ with a short exact sequence
of Hopf algebras
$$
B\rInto T(V)\rOnto E(V).
$$
It follows that the Poincar\'e series
$$
\chi(H_*(\Omega Q))\geq \chi(B)=\frac{\chi(T(V))}{\chi(E(V))}.
$$
From the decomposition
$$
\Omega Y\simeq \Omega Q\times X^{\min}(Y),
$$
the Poincar\'e series
$$
\chi(H_*(X^{\min}(Y)))\leq \chi(E(V)).
$$
Since $X^{\min}(Y)$ is filtered retract of $\Omega Y$, $[\Sigma
X^{\min}(Y)]_1$ is a retract of $[\Sigma \Omega Y]_1=Y$. Because $Y$
is atomic, $[\Sigma X^{\min}(Y)]_1\simeq Y$. It follows that
$$
V\subseteq E^0H_*(X^{\min}(Y)).
$$
Since $X^{\min}(Y)$ is an $H$-space, $E^0H_*(X^{\min}(Y))$ is a Hopf
algebra and so
$$
\chi(H_*(X^{\min}(Y)))=\chi(E^0H_*(X^{\min}(Y)))\geq \chi(E(V)).
$$
Thus $\chi(H_*(X^{\min}(Y)))=\chi(E(V))$ and so
$H_*(X^{\min}(Y))=E(V)$ as coalgebras. It follows that $X^{\min}(Y)$
is an $H$-space having $Y$ as a retractile generating complex and
hence statement (3). The proof is finished.
\end{proof}

\end{document}